\documentclass[11pt]{amsart}
\usepackage{amssymb,enumerate}
\usepackage[colorlinks,bookmarks,citecolor=cyan,linkcolor=purple]{hyperref}
\usepackage{tikz}
\usetikzlibrary{arrows,positioning,shapes}
\usepackage[mathscr]{euscript}
\usepackage[margin=.89in]{geometry}
\usepackage{cleveref,subcaption}

\theoremstyle{definition}
\newtheorem{lemma}{Lemma}[section]
\newtheorem{theorem}[lemma]{Theorem}
\newtheorem{proposition}[lemma]{Proposition}
\newtheorem{corollary}[lemma]{Corollary}

\newtheorem{example}[lemma]{Example}
\newtheorem{remark}[lemma]{Remark}

\newtheorem{definition}[lemma]{Definition}

\newcommand{\F}{F}
\newcommand{\zero}{\hat{0}}			
\newcommand{\A}{\mathscr{A}}			
\newcommand{\B}{\mathscr{B}}

\newcommand{\M}{\mathscr{M}}
\newcommand{\Z}{\mathbb{Z}}			
\newcommand{\R}{\mathbb{R}}			

\newcommand{\rk}{\rho}
\newcommand{\C}{\mathbb{C}}
\newcommand{\G}{G}
\renewcommand{\S}{S}
\newcommand{\cx}{\mathcal{C}}
\newcommand{\x}{\mathrm{x}}
\newcommand{\y}{\mathrm{y}}
\newcommand{\boo}{\mathscr{B}}

\newcommand{\rkG}{\rk_{\G}}

\newcommand{\actson}{\curvearrowright}
\newcommand{\X}{\mathfrak{X}}
\newcommand{\GG}{\mathbb{G}}

\DeclareMathOperator{\at}{A}
\DeclareMathOperator{\Hom}{Hom}
\DeclareMathOperator{\rank}{rk}
\DeclareMathOperator{\cl}{cl}

\newcommand{\st}{\colon}
\newcommand{\into}{\hookrightarrow}		
\newcommand{\ol}[1]{\overline{#1}}
\newcommand{\angles}[1]{\langle #1 \rangle}
\newcommand{\solidnodes}{\tikzstyle{every node}=
[draw,circle,fill=black,minimum size=4pt,inner sep=0pt]}
\newcommand{\labelnodes}{\tikzstyle{every node}=[draw,circle,inner
sep=1pt,scale=.9]}
\newcommand{\opennodes}{\tikzstyle{every node}=[draw,circle,minimum
size=4pt,inner sep=0pt]}

\makeatletter
\newcommand{\mylabel}[2]{(#2)\def\@currentlabel{#2}\label{#1}}
\makeatother

\title{Matroid schemes and geometric posets}
\author{Christin Bibby}
\address{Louisiana State University, Baton Rouge, LA, USA}
\email{\url{bibby@math.lsu.edu}}



\begin{document}

\begin{abstract}
The intersection data of a hyperplane arrangement is described by a geometric
lattice, or equivalently a simple matroid. There is a rich interplay between
this combinatorial structure and the topology of the arrangement complement.
In this paper, we characterize the combinatorial structure underlying
an abelian arrangement (such as a toric or elliptic arrangement) by defining a
class of geometric posets and a generalization of matroids called matroid schemes.
The intersection data of an abelian arrangement is encoded in a geometric poset,
and we prove that a geometric poset is equivalent to a simple matroid scheme.
We lay foundations for the theory of matroid schemes, discussing rank, flats,
and independence. We also extend the definition of the Tutte polynomial to this
setting and prove that it satisfies a deletion-contraction recurrence.
\end{abstract}

\maketitle

\section{Introduction}

A matroid is an important combinatorial object which has been particularly
influential in the theory of hyperplane arrangements, 
where it encodes the intersection data of the hyperplanes in an arrangement.
The interplay between this combinatorial structure and the topology of the
arrangement complement has been a primary theme in arrangement theory for
decades. Recently, arrangement theory has broadened to include \textit{abelian
arrangements} (\Cref{def:arr}), such as toric and elliptic analogues of
hyperplane arrangements.
The intersection data of such an arrangement forms a poset, known as the
\textit{poset of layers} (\Cref{def:layers}), but it is not necessarily a
geometric lattice or even a semilattice, thus not encoded in a matroid.

In this paper, we completely characterize the combinatorial structure of an
abelian arrangement in the same way that matroids and geometric lattices do for
central hyperplane arrangements (see \Cref{fig:bigpicture}). 
More explicitly, we define \textit{matroid schemes}
(Definition \ref{def:matroid scheme}) and \textit{geometric posets} (Definition \ref{def:geom}), 
for which:
\begin{enumerate}[$\bullet$]
\item a geometric poset is equivalent to a simple matroid scheme (Theorem
\ref{thm:GP=MS});
\item the Tutte polynomial of a matroid scheme satisfies a deletion-contraction
recurrence (\Cref{thm:tutte});
\item the intersection data of an abelian arrangement is 
encoded in a matroid scheme
(\Cref{thm:arr}).
\end{enumerate}

While a matroid imposes a \textit{rank} function on a Boolean lattice, we define a
matroid scheme by imposing a similar function on a simplicial poset. 
Since every closed interval in a simplicial poset is a Boolean lattice, we
view a matroid prescheme as made up of \textit{local} matroids (\Cref{def:local}) glued
together in some reasonable manner, and a matroid scheme satisfies an
additional global geometric condition. This models our topological
application, since an abelian arrangement is locally a linear arrangement
(\Cref{def:localarr}). Note, however, that a matroid scheme is 
not a scheme in the usual sense of algebraic geometry.

There are other combinatorial structures that have been associated to toric
arrangements, namely arithmetic matroids \cite{moci,DM}, matroids over $\Z$
\cite{FM}, and semimatroids with a group action \cite{DR}.
Pagaria \cite{pagaria} proved that the poset of layers of a toric
arrangement is not determined by either its arithmetic matroid or its matroid
over $\Z$. Our matroid schemes are more robust, and they do determine the poset
of layers.
Matroid schemes and geometric posets complement the work of \cite{DR,dd} on
semimatroids and geometric semilattices with a group
action, by formalizing the structure of the quotient by the group action. 
Indeed, under mild hypotheses, the orbits of a group action on a geometric
semilattice form a geometric poset (\Cref{thm:Gsemi}). Not every matroid scheme
arises in this way (\Cref{ex:nonpos}),
and
the framework presented here is the first approach to understanding the Dowling
posets of \cite{BG} from a matroid perspective.
Dowling posets generalize partition and Dowling lattices, which are fundamental
examples in matroid theory, and they arise naturally from certain arrangements
of submanifolds. We prove that Dowling posets are geometric
posets (\Cref{thm:dowling}), and as such they provide a plethora of examples of
matroid schemes with topological motivation.

\subsection*{Acknowledgements}
The author was supported by NSF DMS-2204299, and thanks
Emanuele Delucchi for many fruitful conversations.

\begin{figure}[hb]
\begin{tikzpicture}[scale=.9]
\node[rotate=270] at (0,1) {$\subseteq$};
\node[rotate=270] at (0,3) {$\subseteq$};
\node[rotate=270] at (5,1) {$\subseteq$};
\node[rotate=270] at (5,3) {$\subseteq$};
\node[rotate=270] at (10,1) {$\subseteq$};
\node[rotate=270] at (10,3) {$\subseteq$};
\tikzstyle{every node} = [draw,rectangle,fill=none]
\tikzset{>=angle 90,shorten >=5pt,shorten <=5pt}
\node[text width=3.2cm,align=center] (linear) at (0,4) {central hyperplane arrangement};
\node[text width=3.1cm,align=center] (affine) at (0,2) {affine hyperplane arrangement};
\node[text width=3.5cm,align=center] (periodic) at (0,-2) {$\G$-periodic hyperplane arrangement};
\node (abelian) at (0,0) {abelian arrangement};
\node (lattice) at (5,4) {geometric lattice};
\node (semilattice) at (5,2) {geometric semilattice};
\node (matroid) at (10,4) {matroid};
\node (semimatroid) at (10,2) {semimatroid};
\node (geometricposet) at (5,0) {geometric poset};
\node (cwmatroid) at (10,0) {matroid scheme};
\node[text width=2.5cm,align=center] (gsemilattice) at (5,-2) {geometric $\G$-semilattice};
\node (gsemimatroid) at (10,-2) {$\G$-semimatroid};
\draw[->] (linear) edge (lattice);
\draw[->] (affine) edge (semilattice);
\draw[->] (abelian) edge (geometricposet);
\draw[->] (periodic) edge (gsemilattice);
\draw (geometricposet.east) edge[bend right,->] (cwmatroid.west);
\draw (geometricposet.east) edge[bend left,<-] (cwmatroid.west);
\draw (semilattice.east) edge[bend right,->] (semimatroid.west);
\draw (semilattice.east) edge[bend left,<-] (semimatroid.west);
\draw (lattice.east) edge[bend right,->] (matroid.west);
\draw (lattice.east) edge[bend left,<-] (matroid.west);
\draw (gsemilattice.east) edge[bend right,->] (gsemimatroid.west);
\draw (gsemilattice.east) edge[bend left,<-] (gsemimatroid.west);
\tikzstyle{every node} = []
\draw (periodic) edge[->] node[right]{$/\G$} (abelian);
\draw (gsemilattice) edge[->] node[right]{$/\G$} (geometricposet);
\draw (gsemimatroid) edge[->] node[right]{$/\G$} (cwmatroid);
\end{tikzpicture}
\caption{The combinatorial structures underlying 
arrangements.}
\label{fig:bigpicture}
\end{figure}

\tableofcontents

\section{Matroids and semimatroids}

We begin by recalling the definitions of matroid and semimatroid on a ground set
$E$. The most convenient definition of a matroid will be in terms of a rank
function on $\boo(E)$, the collection of subsets of $E$. A semimatroid is a
generalization in which the rank function is defined only on a simplicial
complex $\cx\subseteq\boo(E)$.
\begin{definition}\label{def:mat}
A \textbf{matroid} is a pair $(E,\rk)$ where $E$ is a finite set and
$\rk:\boo(E)\to\Z_{\geq0}$ is a function satisfying:
\begin{enumerate}
\item[\mylabel{m1}{\textbf{R1}}] if $X\subseteq E$, then $0\leq \rk(X)\leq |X|$.
\item[\mylabel{m2}{\textbf{R2}}] if $X\subseteq Y\subseteq E$, then
$\rk(X)\leq\rk(Y)$.
\item[\mylabel{m3}{\textbf{R3}}] if $X,Y\subseteq E$, then 
$\rk(X)+\rk(Y)\geq \rk(X\cup Y) + \rk(X\cap Y)$.
\end{enumerate}
\end{definition}

\begin{definition}\label{def:semi}
A \textbf{semimatroid} is a pair $(\cx,\rk)$ where $\cx$ is a finite simplicial complex and $\rk:\cx\to\Z_{\geq0}$ is a function satisfying:
\begin{enumerate}
\item[\mylabel{s1}{\textbf{S1}}] if $X\in\cx$, then $0\leq\rk(X)\leq |X|$.
\item[\mylabel{s2}{\textbf{S2}}] if $X,Y\in\cx$ and $X\subseteq Y$, then
$\rk(X)\leq\rk(Y)$.
\item[\mylabel{s3}{\textbf{S3}}] if $X,Y\in\cx$ and $X\cup Y\in\cx$, then
$\rk(X)+\rk(Y)\geq\rk(X\cup Y)+\rk(X\cap Y)$.
\item[\mylabel{s4}{\textbf{S4}}] if $X,Y\in\cx$ and $\rk(X)=\rk(X\cap Y)$, then
$X\cup Y\in\cx$.
\item[\mylabel{s5}{\textbf{S5}}] if $X,Y\in\cx$ and $\rk(X)<\rk(Y)$, then $X\cup
\{y\}\in\cx$ for some $y\in Y\setminus X$.
\end{enumerate}
\end{definition}

\begin{remark}
Through our definition, we consider the vertices of the simplicial complex $\cx$
to be the ground set of the semimatroid. Ardila \cite[Def. 2.1]{ardila} gave a
slightly more general definition of semimatroid, allowing the ground set to
contain elements that were not in $\cx$.

Semimatroids are equivalent to the quasi-matroids defined by
Kawahara \cite{kawahara}. The motivation behind both Ardila's and Kawahara's
work was to describe the combinatorial structure of an arrangement of affine
hyperplanes.
\end{remark}

\bigskip

\section{Poset terminology and notation}

Let $P$ be a finite partially ordered set, or \textbf{poset}. Given $x\in P$ and
$T\subseteq P$, we use the following notation:
\[T_{\leq x} := \{y\in T\st y\leq x\} 
\hspace{1cm}
\text{and}
\hspace{1cm}
T_{\geq x} := \{y\in T\st y\geq x\}\]
For a subset $T\subseteq P$, let $\bigvee T$ be the set of minimal upper bounds
and $\bigwedge T$ be the set of maximal lower bounds. That is,
\[\bigvee T := \min\{u\in P\st u\geq a,\ \forall a\in T\} 
\hspace{1cm}\text{and}\hspace{1cm}
\bigwedge T := \max\{\ell\in P\st \ell\leq a,\ \forall a\in T\}.\]
If $T=\{x,y\}$, we write $x\vee y := \bigvee T$ and $x\wedge y := \bigwedge T$. 
The poset $P$ is a \textbf{lattice} if $|x\vee y|=1$ and $|x\wedge y|=1$ for any
$x,y\in P$.
The poset $P$ is a \textbf{meet-semilattice} if $|x\wedge y|=1$ for all $x,y\in
P$.

The poset $P$ is \textbf{bounded below} if it has a unique minimum element,
denoted by $\zero$. The poset $P$ is \textbf{ranked} if there is an
order-preserving map $\rank_P:P\to\Z_{\geq 0}$ such that whenever $y$ covers
$x$, $\rank_P(y)=\rank_P(x)+1$. We may assume that, if $P$ is bounded below,
then $\rank_P(\zero)=0$. If $P$ is bounded below and ranked, let \[\at(P):=
\{a\in P\st \rank_P(a)=1\}\] be its set of \textbf{atoms}. 

Two posets $P$ and $Q$ are isomorphic if there is a bijection
$\phi:P\to Q$ such that $x\leq_Py$ if and only if $\phi(x)\leq_Q\phi(y)$.

For a set $T$, denote by $\boo(T)$ the Boolean algebra on $T$, that is, the
lattice whose elements are subsets of $T$ partially ordered by inclusion.
A poset $P$ is a \textbf{simplicial poset} if it is bounded below, ranked,
and for any $x\in P$, the subposet $P_{\leq x}$ is isomorphic to
the Boolean lattice $\boo(\at(P_{\leq x}))$.
The rank function on a simplicial poset $P$ is determined by
$\rank_P(x)=|\at(P_{\leq x})|$, and we will abbreviate $|x|:=\rank_P(x)$
whenever it is understood that $P$ is a simplicial poset.
The reader is warned that we use this notation liberally and rely on context: if
$x$ is an element of a simplicial poset then $|x|$ means its rank in the
simplicial poset, and if $T$ is a finite set then $|T|$ means the cardinality of
the set (equivalently, its rank in the poset $\boo(T)$). For instance, since we
view $x\vee y$ as a set rather than an element, the notation $|x\vee y|$ means
the cardinality of this set.

Let $P$ be a simplicial poset, and let $x,a\in P$ such that $a\leq x$. Then in
the Boolean lattice $P_{\leq x}$, there is a unique element $x\setminus a$ for
which $(x\setminus a)\wedge a=\zero$ and $(x\setminus a)\vee a=x$. This element
$x\setminus a$ is called the \textbf{complement} of $a$ in $P_{\leq x}$, and it
satisfies $|x\setminus a|=|x|-|a|$.
Further note that for elements $x,y$ in a simplicial poset $P$, if $x\vee
y\neq\emptyset$ then $|x\wedge y|=1$, since any lower bound must be a lower
bound in the lattice $P_{\leq u}$ for any $u\in x\vee y$.

\section{Matroid schemes}
Our generalization of matroid uses a simplicial poset $S$ in place of the
Boolean lattice or simplicial complex of a (semi)matroid. The idea is to place a
matroid structure on each Boolean lattice $S_{\leq x}$ (via
\eqref{x1}--\eqref{x3} below, see also \Cref{def:local}), but with
reasonable compatibility among these ``local'' matroid structures (via
\eqref{x4}), and sometimes a global geometric condition (via \eqref{x5},
compare with \eqref{g2} of \Cref{def:geom}).

\begin{definition}\label{def:matroid scheme}
A \textbf{matroid prescheme} is a pair $(\S,\rk)$, where $\S$ is a finite
simplicial poset and $\rk:\S\to\Z_{\geq0}$ is a function satisfying:
\begin{enumerate}
\item[\mylabel{x1}{\textbf{M1}}]
if $x\in \S$ then $0\leq \rk(x) \leq |x|$.
\item[\mylabel{x2}{\textbf{M2}}]
if $x,y\in \S$ and $x\leq y$, then $\rk(x)\leq\rk(y)$.
\item[\mylabel{x3}{\textbf{M3}}]
if $x,y\in \S$ and $u\in x\vee y$, then
$\rk(x)+\rk(y)\geq\rk(u)+\rk(x\wedge y)$.
\item[\mylabel{x4}{\textbf{M4}}]
if $x,y\in \S$ and $\ell\in x\wedge y$ such that
$\rk(x)=\rk(\ell)$, then $x\vee y\neq\emptyset$.
\end{enumerate}
A \textbf{matroid scheme} is a matroid prescheme $(\S,\rk)$ satisfying
the additional condition:
\begin{enumerate}
\item[\mylabel{x5}{\textbf{M5}}]
if $x,y\in \S$ and $\rk(x)<\rk(y)$, then there is some
$a\in\at(\S)$ such that  $a\leq y$, $a\not\leq x$, and $x\vee a\neq\emptyset$.
\end{enumerate}
Two matroid preschemes $(\S,\rk)$ and $(\S',\rk')$ are \textbf{isomorphic} if there is
an isomorphism $\phi:\S\to\S'$ of simplicial posets such that
$\rk(x)=\rk'(\phi(x))$.

\end{definition}

\begin{remark}
A simplicial complex is equivalent to a simplicial meet-semilattice, for which
axioms \eqref{x1}--\eqref{x5} are a straightforward translation of axioms
\eqref{s1}--\eqref{s5}.
In particular, if $\S$ is a lattice or meet-semilattice, then
\Cref{def:matroid scheme} of a matroid
scheme $(\S,\rk)$ is equivalent to \Cref{def:mat,def:semi}
for a matroid or semimatroid, respectively.
\end{remark}

It is straightforward to check that for any matroid prescheme
$\M=(S,\rk)$ and any $x\in S$,
the restriction of $\rk$ to the Boolean lattice $S_{\leq x}$ satisfies
\eqref{m1}--\eqref{m3}. This yields the following definition for our ``local''
matroids.
\begin{definition}\label{def:local}
Let $\M=(S,\rk)$ be a matroid prescheme and $x\in S$. The \textbf{localization} of $\M$
at $x$ is the matroid $\M_x = (S_{\leq x},\rk|_{S_{\leq x}})$.
\end{definition}

\begin{example}\label{ex:cw}
\Cref{fig:ex} depicts the Hasse diagram of two simplicial posets $S$,
and four examples of a function $\rk:S\to\Z_{\geq0}$
indicated by labels at each $x\in S$. 
Examples \eqref{fig:ex1} and \eqref{fig:ex2} are both matroid schemes.
Notice that \eqref{fig:ex2} 
is the only possible
matroid scheme structure on that particular simplicial poset. Indeed, no atom $a$
can have $\rk(a)=0$ by \eqref{x4}, and if a maximal element $u$ had $\rk(u)=2$
it would form a matroid prescheme but not a matroid scheme (see
\eqref{fig:nox5}).

An example which satisfies \eqref{x1}--\eqref{x3} but fails \eqref{x4},
hence is not a matroid prescheme, is depicted in \eqref{fig:nox4}.
This example also demonstrates the necessity of \eqref{x4} in ensuring
that the closure operator (\Cref{def:closure}) and flats 
(\Cref{def:flats}) are well-defined.

\begin{figure}[ht]
\begin{subfigure}[t]{.2\textwidth}
\centering
\begin{tikzpicture}[scale=.8]
\labelnodes
\node (bot) at (0,0.3) {0};
\foreach \x in {-1,0,1} {
\node (\x) at (\x*1.5,1.5) {1};
\node (\x') at (0.75*\x-0.75,3) {2};
};
\node (a) at (1.5,3) {2};
\node (top) at (-.75,4.2) {2};
\draw[-] (bot)--(-1)--(-1')--(top)--(0')--(-1);
\draw[-] (bot)--(0)--(a)--(1)--(bot);
\draw[-] (-1')--(0)--(1')--(1)--(0');
\draw[-] (1')--(top);
\end{tikzpicture}
\caption{matroid scheme}
\label{fig:ex1}
\end{subfigure}
\hfill
\begin{subfigure}[t]{.2\textwidth}
\centering
\begin{tikzpicture}[scale=.9]
\labelnodes
\foreach \x in {1,2,3,4} {
\node (\x) at (\x-2.5,1.5) {1};
}
\node (0) at (0,0.3) {0};
\node (1') at (-1,2.5) {1};
\node (4') at (1,2.5) {1};
\draw[-] (0)--(1)--(1')--(2)--(0)--(3)--(4')--(4)--(0);
\end{tikzpicture}
\caption{matroid scheme}
\label{fig:ex2}
\end{subfigure}
\hfill
\begin{subfigure}[t]{.2\textwidth}
\centering
\begin{tikzpicture}[scale=.9]
\node at (-1.3,2.7) {$x$};
\node at (1.3,2.7) {$y$};
\node at (0.3,0.3) {\phantom{$\ell$}};
\labelnodes
\foreach \x in {1,2,3,4} {
\node (\x) at (\x-2.5,1.5) {1};
}
\node (0) at (0,0.3) {0};
\node (1') at (-1,2.5) {1};
\node (4') at (1,2.5) {2};
\draw[-] (0)--(1)--(1')--(2)--(0)--(3)--(4')--(4)--(0);
\end{tikzpicture}
\caption{\eqref{x5} fails}
\label{fig:nox5}
\end{subfigure}
\hfill
\begin{subfigure}[t]{.2\textwidth}
\centering
\begin{tikzpicture}[scale=.9]
\node at (-1.3,2.7) {$x$};
\node at (1.3,2.7) {$y$};
\node at (0.4,0.3) {$\ell$};
\labelnodes
\foreach \x in {1,2,3,4} {
\node (\x) at (\x-2.5,1.5) {0};
}
\node (0) at (0,0.3) {0};
\node (1') at (-1,2.5) {0};
\node (4') at (1,2.5) {0};
\draw[-] (0)--(1)--(1')--(2)--(0)--(3)--(4')--(4)--(0);
\end{tikzpicture}
\caption{\eqref{x4} fails}
\label{fig:nox4}
\end{subfigure}

\caption{
Depicted are Hasse diagrams for simplicial posets, and each element $x$
labelled by $\rk(x)$. The first two are matroid schemes; the third is a
matroid prescheme; the fourth is neither. See \Cref{ex:cw}.}
\label{fig:ex}
\end{figure}

\end{example}

We will need the following useful properties regarding the rank function of a
matroid scheme, the latter giving a ``purity'' condition of the rank function.

\begin{proposition}\label{prop:rk}
Let $\M=(S,\rk)$ be a  matroid prescheme.
\begin{enumerate}
\item\label{item:joinatom} If $x\in S$, $a\in\at(S)$, and $u\in x\vee a$, then
$\rk(x)\leq\rk(u)\leq\rk(x)+1$.
\end{enumerate}
Now assume that $\M=(\S,\rk)$ is a matroid scheme.
\begin{enumerate}
\setcounter{enumi}{1}
\item\label{item:equirkjoin} 
If $T\subseteq S$ and $u,v\in\bigvee T$, then $\rk(u)=\rk(v)$.
\item\label{item:equirkmax} The function $\rk$ is constant on $\max(S)$.
\end{enumerate}
\end{proposition}
\begin{proof}
For \eqref{item:joinatom}, if $a\leq x$, then $u=x$ and hence $\rk(u)=\rk(x)$.
Otherwise, we have $a\wedge x=\zero$ and hence
\[\rk(x)\leq \rk(u)=\rk(u)+\rk(\zero) \leq \rk(x)+\rk(a)\leq \rk(x)+1,\]
by \eqref{x2}, \eqref{x1}, \eqref{x3}, and \eqref{x1}, respectively.

Now for \eqref{item:equirkjoin}, suppose $u,v\in \bigvee T$. Since $S$ is
simplicial, we have $\at(S_{\leq u})=\at(S_{\leq v})$ because both sets are
equal to $\{a\in\at(S)\st a\leq t, \text{ some }t\in T\}$.
If $\rk(u)<\rk(v)$, then \eqref{x5} would imply existence of an atom $a\in
\at(S_{\leq v})\setminus\at(S_{\leq u})$, which cannot exist. Thus,
$\rk(u)=\rk(v)$.

Finally, for \eqref{item:equirkmax},
if $u,v\in\max(S)$ such that $\rk(u)<\rk(v)$, then again by \eqref{x5}, there is
some $a\in A(\S)$ with $a\leq v$, $a\not\leq u$, $a\vee u\neq\emptyset$. This
contradicts maximality of $u$, thus $\rk(u)=\rk(v)$.
\end{proof}

\begin{definition}
Say that a matroid prescheme $\M=(\S,\rk)$ is \textbf{pure} if $\rk$ is
constant on $\max(\S)$. When $\M$ is pure, define its \textbf{rank}
$\rk(\M)$ to be the rank of any maximal element of $\S$.
\end{definition}

\begin{example}
By \Cref{prop:rk}\eqref{item:equirkmax}, any matroid scheme is pure.

The matroid prescheme depicted in \Cref{fig:nox5} is not pure.
\end{example}

As a last result in this section, we will establish stronger versions
\eqref{x4} and \eqref{x5}, analogous to \cite[(CR1'),(CR2')]{ardila} for
semimatroids.
\begin{proposition}
A matroid prescheme $\M=(\S,\rk)$ satisfies the condition:
\begin{enumerate}
\item[\mylabel{x4'}{\textbf{M4'}}]
if $x,y\in \S$ and $\ell\in x\wedge y$ such that
$\rk(x)=\rk(\ell)$, then $x\vee y\neq\emptyset$ and $\rk(u)=\rk(y)$ for
any $u\in x\vee y$.
\end{enumerate}
A matroid scheme $\M=(\S,\rk)$ satisfies the condition:
\begin{enumerate}
\item[\mylabel{x5'}{\textbf{M5'}}]
if $x,y\in \S$ and $\rk(x)<\rk(y)$, then there is some
$a\in\at(\S)$ such that  $a\leq y$, $a\not\leq x$, $x\vee
a\neq\emptyset$, and $\rk(x\vee a)=\rk(x)+1$.
\end{enumerate}
\end{proposition}
\begin{proof}
For the conclusion on rank in \eqref{x4'}, we have $\rk(u)\geq\rk(y)$ by
\eqref{x2} and, since $\rk(\ell)=\rk(x)$, we have $\rk(y)\leq \rk(u)$ by
\eqref{x3}.

For \eqref{x5'}, we repeatedly apply \eqref{x5} until we have a
chain $x=w_0<w_1<w_2<\cdots<w_k$ with $w_i\setminus
w_{i-1}=a_i\in\at(\S_{\leq y})$ and $\rk(w_k)=\rk(y)$.
There exists $z_i\in a_i\vee x$ with $z_i\leq w_i$ for each $i$, and we
claim that $\rk(z_i)=\rk(x)+1$ for some $i$.
We argue by contradiction, assuming that $\rk(z_i)=\rk(x)$ for all $i$.
Then since $w_i\in w_{i-1}\vee z_i$ and $x=w_{i-1}\wedge z_i$, we have
by \eqref{x4'} that $\rk(w_i)=\rk(w_{i-1})$ for all $i$.
But this would imply $\rk(x)=\rk(y)$, our contradiction.
\end{proof}

\section{Closure and flats}
The lattice of flats of a matroid plays an important role in arrangement theory,
and it gives rise to 
an equivalent definition of a simple matroid. In order to see that analogous
story in our setting, we define here the closure operator and poset of
flats for a matroid prescheme.

\begin{lemma}
Let $(\S,\rk)$ be a matroid prescheme, and let $r\in\Z_{\geq0}$.
If $x$ and $y$ are distinct maximal elements of $\rk^{-1}(r)$, and
$\ell\in x\wedge y$, then $\rk(\ell)<r$.
\end{lemma}
\begin{proof}
Suppose not, that is, $\rk(\ell)=\rk(x)$. Then \eqref{x4} implies $x\vee
y\neq\emptyset$, so
let $u\in x\vee y$. Then by \eqref{x2}, $r=\rk(x)\leq \rk(u)$, and by \eqref{x3},
$\rk(u)\leq \rk(x)+\rk(y)-\rk(\ell) = r$. Thus, $u\in \rk^{-1}(r)$, contradicting
maximality of $x$ and $y$.
\end{proof}

This lemma implies that the following is a well-defined function $\cl:S\to S$.
We will see in \Cref{prop:closure} that $\cl$ is a closure operator, meaning
that it satisifes \eqref{cl1}--\eqref{cl3} below. It also satisfies a
``geometric'' property \eqref{cl4} which directly translates that of matroids
when $S$ is a lattice.

\begin{definition}\label{def:closure}
Let $\M=(S,\rk)$ be a matroid prescheme.
For $x\in\S$, define the \textbf{closure} of $x$, denoted by $\cl(x)$, to be the
unique maximal element of  the set
$\{y\in S_{\geq x} \st \rk(y)=\rk(x)\}$.
\end{definition}

\begin{lemma}\label{lem:closure}
Let $(S,\rk)$ be a matroid prescheme, and let $x,y\in S$. Then $y\leq\cl(x)$ if and
only if there exists $u\in x\vee y$ such that $\rk(u)=\rk(x)$.
\end{lemma}
\begin{proof}
If $y\leq \cl(x)$, then $\cl(x)$ is a common upper bound of $x$ and $y$. This
implies that there is some $u\in x\vee y$ such that $x\leq u\leq \cl(x)$. But
then \eqref{x2} implies $\rk(x)\leq \rk(u)\leq\rk(\cl(x))=\rk(x)$, so
$\rk(u)=\rk(x)$.

Conversely, if $u\in x\vee y$ and $\rk(u)=\rk(x)$, then $y\leq u \leq\cl(x)$ by
maximality of $\cl(x)$.
\end{proof}

\begin{proposition}\label{prop:closure}
Let $(S,\rk)$ be a matroid prescheme. The function
$\cl:S\to S$ satisfies the following properties:
\begin{enumerate}
\item[\mylabel{cl1}{\textbf{CL1}}] if $x\in S$, then $x\leq \cl(x)$.
\item[\mylabel{cl2}{\textbf{CL2}}] if $x,y\in S$ and $x\leq y$, then
$\cl(x)\leq \cl(y)$.
\item[\mylabel{cl3}{\textbf{CL3}}] if $x\in S$, then $\cl(\cl(x)) = \cl(x)$.
\item[\mylabel{cl4}{\textbf{CL4}}] if $x\in S$, $a,b\in\at(S)$, $u\in x\vee
b$, $a\leq \cl(u)$, and $a\not\leq\cl(x)$, then there exists $v\in x\vee a$ such
that $v\leq\cl(u)$ and $b\leq \cl(v)$. 
\end{enumerate}
\end{proposition}
\begin{proof}
Properties \eqref{cl1}--\eqref{cl3} are immediate.
To see \eqref{cl4}, let $a$, $x$, and $b$ be as given. Then $\cl(u)$ is a common
upper bound of these three, so there exists $v\in x\vee a$ and $w\in x\vee
a\vee b$ such that 
$v\leq w \leq \cl(u)$. Necessarily, $u\leq w\leq \cl(u)$ and hence
$\cl(w)=\cl(u)$. By \Cref{prop:rk}\eqref{item:joinatom}, we have $\rk(x)\leq
\rk(v)\leq \rk(x)+1$. 
Since $v\geq x$, $\rk(x)=\rk(v)$ would imply $\cl(x)=\cl(v)\geq a$.
But $a\not\leq \cl(x)$ by assumption, so we must have $\rk(v)=\rk(x)+1$.
Similarly, $\rk(u)=\rk(x)+1$, and hence $\rk(v)=\rk(u)=\rk(w)$ implying $b\leq
\cl(v)$ by \Cref{lem:closure}.
\end{proof}

\begin{definition}\label{def:flats}
Let $\M=(\S,\rk)$ be a matroid prescheme.
We say that $x\in\S$ is a \textbf{flat} if $\cl(x)=x$.
The \textbf{poset of flats} $\F(\M)$ is the subposet of $\S$ whose elements are
the flats of $\M$.
\end{definition}

\begin{example}\label{ex:flats}
The flats of the matroid preschemes in \Cref{ex:cw} are depicted in
\Cref{fig:flats}. 
\begin{figure}[ht]
\begin{subfigure}[t]{.3\textwidth}
\centering
\begin{tikzpicture}
\solidnodes
\node (0) at (0,0.5) {};
\foreach \x in {1,2,3} {
\node (\x) at (\x-2,1.5) {};
}
\node (u) at (-.5,2.5) {};
\node (v) at (.5,2.5) {};
\draw[-] (0)--(1)--(u)--(2)--(0)--(3)--(v)--(2);
\draw[-] (u)--(3);
\end{tikzpicture}
\caption{ }
\label{fig:flats1}
\end{subfigure}
\hspace{5mm}
\begin{subfigure}[t]{.3\textwidth}
\centering
\begin{tikzpicture}
\solidnodes
\node (0) at (0,0.5) {};
\node (1) at (-0.5,1.5) {};
\node (2) at (0.5,1.5) {};
\draw[-] (1)--(0)--(2);
\end{tikzpicture}
\caption{ }
\label{fig:flats2}
\end{subfigure}
\hspace{5mm}
\begin{subfigure}[t]{.3\textwidth}
\centering
\begin{tikzpicture}
\node at (-1.2,1.2) {$x$};
\node at (.8,2.5) {$y$};
\solidnodes
\node (0) at (0,0.5) {};
\foreach \x in {1,2,3} {
\node (\x) at (\x-2,1.5) {};
}
\node (v) at (.5,2.5) {};
\draw[-] (0)--(1);
\draw[-] (2)--(0)--(3)--(v)--(2);
\end{tikzpicture}
\caption{ }
\label{fig:flats3}
\end{subfigure}
\caption{The posets of flats for the matroid schemes depicted in
\Cref{fig:ex1,fig:ex2,fig:nox5}.}
\label{fig:flats}
\end{figure}
\end{example}

\begin{proposition}\label{prop:localflats}
If $\M=(S,\rk)$ is a matroid prescheme and $x\in F(\M)$ is a flat of $\M$, then
the poset of flats of the local matroid $\M_x$ is $F(\M_x)=F(\M)_{\leq x}$.
\end{proposition}
\begin{proof}
Since $x$ is a flat, the image of the restriction $\cl|_{S_{\leq x}}:S_{\leq
x}\to S$ is contained in $S_{\leq x}$.
\end{proof}

\begin{lemma}\label{lem:closure2}
Suppose that $\M=(S,\rk)$ is a matroid prescheme, $x,y\in S$, and 
$u\in F(\M)$ such that $u\in \cl(x)\vee_F\cl(y)$. Then there
exists some $v\in x\vee_S y$ such that $\cl(v)=u$.
\end{lemma}
\begin{proof}
Since $u$ is a common upper bound of $x$ and $y$, there exists $v\in x\vee_S y$
such that $v\leq u$ and hence $\cl(v)\leq \cl(u)=u$.
Now, $x\leq v$ implies $\cl(x)\leq \cl(v)$, and similarly $\cl(y)\leq \cl(v)$.
By minimality of $u\in \cl(x)\vee_F\cl(y)$ we must have $\cl(v)=u$.
\end{proof}

\section{Independence}
Matroids are well known for their many cryptomorphic definitions. In this
section, we establish a cryptomorphic definition of matroid (pre)scheme in terms of
independence, which then leads to bases and circuits in the next
sections.
See an example of independence in \Cref{fig:ind}.

\begin{figure}[ht]
\begin{subfigure}[t]{.3\textwidth}
\centering
\begin{tikzpicture}[scale=.8]
\solidnodes
\node (bot) at (0,0.3) {};
\foreach \x in {-1,0,1} {
\node (\x) at (\x*1.5,1.5) {};
\node (\x') at (0.75*\x-0.75,3) {};
}
\node (a) at (1.5,3) {};
\opennodes
\node (top) at (-.75,4.2) {};
\draw[-] (bot)--(-1)--(-1');
\draw[-] (bot)--(0)--(a)--(1)--(bot);
\draw[-] (-1')--(0)--(1')--(1)--(0')--(-1);
\draw[dashed,-] (1')--(top)--(-1');
\draw[dashed,-] (0')--(top);
\end{tikzpicture}
\caption{$I(\M)$}
\label{fig:ind}
\end{subfigure}
\begin{subfigure}[t]{.3\textwidth}
\centering
\begin{tikzpicture}[scale=.8]
\opennodes
\node (bot) at (0,0.3) {};
\foreach \x in {-1,0,1} {
\node (\x) at (\x*1.5,1.5) {};
}
\node (top) at (-.75,4.2) {};
\solidnodes
\foreach \x in {-1,0,1} {
\node (\x') at (0.75*\x-0.75,3) {};
}
\node (a) at (1.5,3) {};
\draw[dashed,-] (bot)--(-1)--(-1')--(top)--(0')--(-1);
\draw[dashed,-] (bot)--(0)--(a)--(1)--(bot);
\draw[dashed,-] (-1')--(0)--(1')--(1)--(0');
\draw[dashed,-] (1')--(top);
\end{tikzpicture}
\caption{$B(\M)$}
\label{fig:basis}
\end{subfigure}
\begin{subfigure}[t]{.3\textwidth}
\centering
\begin{tikzpicture}[scale=.8]
\opennodes
\node (bot) at (0,0.3) {};
\foreach \x in {-1,0,1} {
\node (\x) at (\x*1.5,1.5) {};
\node (\x') at (0.75*\x-0.75,3) {};
}
\node (a) at (1.5,3) {};
\solidnodes
\node (top) at (-.75,4.2) {};
\draw[dashed,-] (bot)--(-1)--(-1')--(top)--(0')--(-1);
\draw[dashed,-] (bot)--(0)--(a)--(1)--(bot);
\draw[dashed,-] (-1')--(0)--(1')--(1)--(0');
\draw[dashed,-] (1')--(top);
\end{tikzpicture}
\caption{$C(\M)$}
\label{fig:circ}
\end{subfigure}
\caption{The independence poset, bases, and circuits for the matroid
scheme $\M$ from \Cref{fig:ex1}.}
\end{figure}

\begin{definition}
Let $\M=(S,\rk)$ be a matroid prescheme. Say $x\in S$ is \textbf{independent}
if $\rk(x)=|x|$, and otherwise $x$ is \textbf{dependent}. The
\textbf{independence poset} $I(\M)$ is the subposet of $S$ whose elements are
the independent elements of $S$.
\end{definition}

Using this definition, we obtain the following cryptomorphic definition of a
matroid (pre)scheme in terms of independence. 
When $S$ is a lattice, \eqref{i1}--\eqref{i3} directly translate to the
usual axioms for the independence complex of a matroid and \eqref{i4}
always holds. In our setting, \eqref{i1}-\eqref{i3} provides a local
matroid structure and \eqref{i4} gives reasonable compatibility
(analogous to \eqref{x4}).
We also have both a local \eqref{i3} and global \eqref{i3'} version of the
independence augmentation axiom, separating matroid preschemes and matroid
schemes.

\begin{theorem}
Let $S$ be a simplicial poset and let $I\subseteq S$.  Then $I$ is the
independence poset of a matroid prescheme on $S$ if and only if $I$ satisfies
the following properties:
\begin{enumerate}
\item[\mylabel{i1}{\textbf{I1}}] $I$ is nonempty.
\item[\mylabel{i2}{\textbf{I2}}] if $x,y\in S$ with $x\leq y$ and $y\in I$,
then $x\in I$.
\item[\mylabel{i3}{\textbf{I3}}] if $x,y\in I$, $|x|<|y|$, and $u\in
x\vee y$, then there is
some $a\in\at(S)$ such that $a\leq y$, $a\not\leq x$, and 
$(a\vee x)_{\leq u}\subseteq I$.
\item[\mylabel{i4}{\textbf{I4}}] if $x,y\in S$, $z\in\max I_{\leq x}$
and $z\leq y$, then $x\vee y\neq\emptyset$.
\end{enumerate}
Moreover, $I$ is the independence poset of a matroid scheme on $S$ if
and only if it satisfies the above properties with \eqref{i3} replaced
by the stronger condition \eqref{i3'}.
\begin{enumerate}
\item[\mylabel{i3'}{\textbf{I3'}}] if $x,y\in I$ and $|x|<|y|$, then there is
some $a\in\at(S)$ such that $a\leq y$, $a\not\leq x$, and $a\vee x$ is a
nonempty subset of $I$.
\end{enumerate}
\end{theorem}
\begin{proof}
First assume that $\M=(\S,\rk)$ is a matroid prescheme, and let
$I=I(\M)$ be its independence poset. Notice that for every $u\in\S$,
$I_{\leq u}$ is the independence poset of the local matroid $\M_u$ and
necessarily satisfies \eqref{i1}--\eqref{i3}. Since these
conditions are local in nature, $I$ must satisfy \eqref{i1}--\eqref{i3}.
To see \eqref{i4}, let $x,y,z$ be as stated. Then there is some $\ell\in x\wedge
y$ such that $z\leq \ell$. Since $z$ is maximally independent under $x$, we
have $\rk(x)=\rk(z)\leq \rk(\ell)\leq\rk(x)$, thus $x\vee y\neq\emptyset$ by
\eqref{x4}.

To establish \eqref{i3'} when $\M$ additionally satisfies \eqref{x5},
let $x,y\in I(\M)$ with $|x|<|y|$. By \eqref{x5'}, there is
some $a\in \at(S)$ such that $a\leq y$, $a\not\leq x$, and $a\vee
x\neq\emptyset$ with $\rk(a\vee x)=\rk(x)+1$, thus $a\vee x$ is a
nonempty subset of $I(\M)$.

Conversely, suppose that $I$ satisfies \eqref{i1}--\eqref{i4}, and define
$\rk:S\to\Z_{\geq 0}$ by $\rk(x)=\max\{|z|\st z\in I_{\leq x}\}$. We claim that
$(S,\rk)$ is a matroid prescheme; since $I=\{x\in S\st \rk(x)=|x|\}$, $I$ will be
its independence poset.
First note that for any $u\in S$, properties
\eqref{i1}--\eqref{i3} imply that $I_{\leq u}$ is the independence poset of a
matroid $(S_{\leq u},\rk|_{S_{\leq u}})$. This implies the local properties
\eqref{x1}--\eqref{x3}. 
To see \eqref{x4}, suppose that $\ell\in x\wedge y$ such that
$\rk(\ell)=\rk(x)$, and let $z\in\max I_{\leq \ell}$. Then $z\leq y$, and since
$\rk(z)=\rk(\ell)=|z|$, we have $z\in\max I_{\leq x}$. Thus $x\vee y\neq\emptyset$
by \eqref{i4}.

Finally, to establish \eqref{x5} when $I$ satisfies \eqref{i3'},
suppose that $x,y\in S$ with $\rk(x)<\rk(y)$. Let $z\in\max I_{\leq x}$
and $w\in\max I_{\leq y}$ so that $|z|=\rk(x)<\rk(y)=|w|$. Then
\eqref{i3'} implies there is some $a\in\at(S)$
such that $a\leq w$, $a\not\leq z$, and $a\vee z$ is a nonempty subset of $I$.
Note that $a\leq w$ implies that $a\leq y$, and maximality of $z$ in $I_{\leq
x}$ implies that $a\not\leq x$. For $u\in a\vee z$, \eqref{i4}
implies that $x\vee u\neq\emptyset$, and since $a\leq u$ this means $x\vee
a\neq\emptyset$.
\end{proof}

\section{Bases}
With terminology motivated by linear algebra, a maximally independent
element is called a basis. 
See \Cref{fig:basis} for an example.
We obtain here a cryptomorphic definition of
a matroid (pre)scheme in terms of its bases.

\begin{definition}
In a matroid prescheme $\M=(S,\rk)$, let \[B(\M)=\max I(\M)=
\max\{x\in S\st \rk(x)=|x|\},\] the
maximally independent elements of $S$. 
An element of $B(\M)$ is called a \textbf{basis}.
\end{definition}

\begin{corollary}\label{cor:basisrank}
If $\M$ is a pure matroid prescheme (such as a matroid scheme), then
$\rk(\M)=\rk(x)$ for any basis $x\in B(\M)$.
\end{corollary}
\begin{proof}
Let $x\in B(\M)$ and $u\in\max S_{\geq x}$.
Then $\rk(u)=\max\{|z|\st z\in I(\M)_{\leq u}\} = |x|=\rk(x)$.
\end{proof}

As with independence, the set of bases should locally look like matroid
bases  (as in \Cref{lem:bases} or \eqref{b1}--\eqref{b3}), but satisfy
reasonable compactibility \eqref{b4}. 
Matroid schemes also satsify a global basis exchange axiom \eqref{b3'}.

\begin{proposition}\label{lem:bases}
The set of bases $B(\M)$ for a matroid prescheme $\M=(S,\rk)$ is equal
to the disjoint union of the local basis sets $B(\M_x)$ where $x$ ranges
over the maximal elements of $S$.
\end{proposition}
\begin{proof}
First of all, the local basis sets are disjoint:
if $b$ is a basis of both $\M_x$ and $\M_y$,
then $x$ and $y$ would have a common upper bound by \eqref{i4}, 
thus  by maximality $x=y$. 
We prove that for any $x\in\max\S$,
$\max(I(\M))_{\leq x} = \max(I(\M)_{\leq x})$, implying $B(\M)_{\leq
x}=B(\M_x)$.
The containment ``$\subseteq$'' is immediate.
For the other direction, suppose that
$z\in\max(I(\M)_{\leq x})$. We must have $z\leq y$ for some basis $y\in
B(\M)$. Then by \eqref{i4}, $x\vee y\neq\emptyset$, and maximality of
$x$ then implies $y\leq x$. But then maximality of $z$ implies $z=y\in
B(\M)_{\leq x}$.
\end{proof}

\begin{theorem}
Let $S$ be a simplicial poset and let $B\subseteq S$. 
Then $B$ is the set of bases of a matroid prescheme on $S$ if and only
if $B$ satisfies the following properties:
\begin{enumerate}
\item[\mylabel{b1}{\textbf{B1}}] $B$ is nonempty.
\item[\mylabel{b2}{\textbf{B2}}] if $x,y\in B$ and $x\leq y$, then
$x=y$.
\item[\mylabel{b3}{\textbf{B3}}]
if $x,y\in B$, $u\in x\vee y$, and $a\in\at(S)$ with $a\leq x$, then
there exists $b\in\at(S)$ such that $b\leq y$ and 
$((x\setminus a)\vee b)_{\leq u}\subseteq B$.
\item[\mylabel{b4}{\textbf{B4}}] if $x,y\in S$, $z\in\max\{z'\in S \st
z'\leq x \text{ and } z'\leq b \text{ for some } b\in B\}$ and $z\leq
y$, then $x\vee y\neq\emptyset$.
\end{enumerate}
Moreover, $B$ is the set of bases for a matroid scheme on $S$ if and
only if $B$ satisfies the above properties with \eqref{b3} replaced by
the stronger condition:
\begin{enumerate}
\item[\mylabel{b3'}{\textbf{B3'}}] 
if $x,y\in B$ and $a\in\at(S)$ with $a\leq x$, then there exists
$b\in\at(S)$ such that $b\leq y$ and $(x\setminus a)\vee b$ is a
nonempty subset of $B$.
\end{enumerate}
\end{theorem}
\begin{proof}
First, let $\M=(S,\rk)$ be a matroid prescheme and $B=B(\M)$ its set of
bases. Properties \eqref{b1}--\eqref{b3} are local in nature; it
suffices to check for $B(\M)_{\leq m}$ for all $m\in\max S$.
This follows from \Cref{lem:bases} since $B_{\leq m}$ is the set of
bases for the local matroid $\M_m$.
Additionally, \eqref{b4} follows from \eqref{i4} since $\{z'\in S \st
z'\leq x, z\leq b \text{ for some } b\in B\} = I(\M)_{\leq x}$.

To establish \eqref{b3'} when $\M$ is a matroid scheme, 
let $x,y,a$ be as stated. Then
\eqref{i3'} implies that there is some $b\in\at(S)$ with $b\leq y$ and
$(x\setminus a)\vee b$ a nonempty subset of $I(\M)$. Any
$u\in (x\setminus a)\vee b$ must be maximal in $I(\M)$, i.e. $u\in
B(\M)$,  since $|u|=|x|=\rk(\M)$ by \Cref{cor:basisrank}.

Conversely, suppose that $B$ satisfies \eqref{b1}--\eqref{b4}, and let
$I=\{x\in S \st x\leq b \text{ for some } b\in B\}$. 
Then for any $m\in\max S$, $B_{\leq m}$ is the collection of bases for a
matroid on $S_{\leq m}$ with independence set $I_{\leq m}$. 
Conditions \eqref{i1}--\eqref{i3} follow.
Additionally, \eqref{b4} implies \eqref{i4} since $I_{\leq x} = \{z'\in
S\st z'\leq x, z\leq b \text{ for some } b\in B\}$.

Now assume further that \eqref{b3'} holds. Before proving \eqref{i3'},
we will first show that for any $x,y\in B$ we have $|x|=|y|$.
Suppose not, and among all pairs $(x,y)$ with $|x|>|y|$, pick one such
that $|x\setminus \ell|$, where $\ell\in x\wedge y$, is minimal. 
Then choose $a\in\at(S)$ with $a\leq x\setminus \ell$ (note that one
exists, since $x\not\leq y$). By \eqref{b3'},
there exists an atom $b\leq y$ with $(x\setminus a)\vee b$ a nonempty
subset of $B$. 
Notice that $b\not\leq x$, since otherwise we'd have some $u\in
(x\setminus a)\vee b$ with $u\leq x$, implying $u=x$ by \eqref{b2} and
thus $b=a\leq x \setminus \ell$ contradicting $b\leq y$.
Now let $u\in (x\setminus a)\vee b$ and $\ell'\in u\wedge y$ such that
$\ell'\geq \ell$. Since $\ell'\geq b$ as well, we have 
$|\ell'|=|\ell|+1$. But
then $|u|=|x|>|y|$ while
$|u\setminus \ell'|<|x\setminus\ell|$, contradicting our
minimality assumption.

Now, we argue for \eqref{i3'} by contradiction. Suppose 
that $x,y\in I$ with $|x|<|y|$ such that for all atoms $a\leq y$,
$a\not\leq x$, we have either $a\vee x=\emptyset$ or $a\vee x$ contains
an element not in $I$. 
Using \eqref{i4}, our assumption means
 that $(a\vee x)\cap I=\emptyset$ for all atoms $a\leq y$.
By definition of $I$, we can find $v,w\in B$ with $v\geq x$ and $w\geq
y$. We may assume that $w$ is chosen so that the size of $A(w) =
\{a\in\at(S) \st a\leq w, a\not\leq y, a\not\leq v\}$ is minimal.
First suppose that $a\in A(w)$. Then by \eqref{b3'} we have $b\in\at(S)$
and $z\in B$ such that $b\leq y$ and $z\in(w\setminus a)\vee b$.
But then $A(z)=A(w)\setminus\{a\}$, contradicting minimality of
$|A(w)|$. We conclude that $A(w)=\emptyset$. 

Let $A'=\{a'\in\at(S)\st a'\leq w\setminus y\}$. Since $A(w)=\emptyset$,
$v$ is an upper bound of $A'$ and hence there is some $u\in\bigvee A'$
such that $u\leq v$. Then $|x|+|u|=|x|+|w|-|y|<|w|=|v|$, so there must
be an atom $a\leq v$ with $a\not\leq x$ and $a\notin A'$.
By \eqref{b3'}, there exists an atom $b\leq w$ with $(v\setminus a)\vee
b$ a nonempty subset of $B$. Since $a\notin A'$, we cannot have $b\in
A'$. This means $b\leq y$, and since $x\leq v\setminus a$ this
contradicts our assumption that $(x\vee b)\cap I=\emptyset$.
\end{proof}

\section{Circuits}
Another common cryptomorphic definition of matroids is through its
circuits, which are minimally dependent. 
See \Cref{fig:circ} for an example.
Here, we similarly characterize the circuits of a matroid prescheme.
As usual, there are conditions \eqref{c1}--\eqref{c3} corresponding to a
local matroid structure and a compatibility condition \eqref{c4}
analogous to \eqref{x4}.

\begin{definition}
In a matroid prescheme $\M=(S,\rk)$, let $C(\M) = \min\{x\in S \st
\rk(x)<|x|\}$, the minimally dependent elements of $S$. An element of $C(\M)$
is called a \textbf{circuit}.
\end{definition}

\begin{remark}
Equivalently, $C(\M)=\min(S\setminus I(\M))$. In particular, if $x\in C(\M)$
and $y<x$ then $\rk(y)=|y|$. Moreover, any $x\in C(\M)$ has
$\rk(x)=|x|-1$.
\end{remark}

\begin{theorem}
Let $S$ be a simplicial poset and let $C\subseteq S$. Then $C$ is the
set of circuits for a matroid prescheme on $S$ if and only if $C$
satisfies the following properties:
\begin{enumerate}
\item[\mylabel{c1}{\textbf{C1}}]
$\zero\notin C$.
\item[\mylabel{c2}{\textbf{C2}}]
if $x,y\in C$ and $x\leq y$, then $x=y$.
\item[\mylabel{c3}{\textbf{C3}}]
if $x,y\in C$ with $x\neq y$ and $u\in x\vee y$, and $a\in\at(S)$ with
$a\leq x\wedge y$, then there exists $z\in C$ such that $z\leq u$ and
$z\not\geq a$.
\item[\mylabel{c4}{\textbf{C4}}]
if $x,y\in S$,
$z\in\max\{z'\in S \st z'\leq x, z'\not\geq c\text{ for any } c\in C\}$,
 and $z\leq y$, then $x\vee y\neq\emptyset$.
\end{enumerate}
\end{theorem}
\begin{proof}
First assume that $\M=(S,\rk)$ is a matroid prescheme, and let $C=C(\M)$
be its set of circuits.
Then \eqref{i1} implies \eqref{c1}, \eqref{c2} holds by minimality
in the definition of $C(\M)$, and
\eqref{c4} follows from \eqref{i4} since
$\{z'\st z'\leq x, z'\not\geq \text{ for any } c\in C\}=I(\M)_{\leq x}$.

To prove that \eqref{c3} holds for $C(\M)$, we recall our notation
for the complements $u\setminus a$, $x\setminus a$, $y\setminus a$, and
$(x\wedge y)\setminus a$ of $a$ in the Boolean lattices $S_{\leq u}$, $S_{\leq
x}$, $S_{\leq y}$, and $S_{\leq x\wedge y}$, respectively.
Then
\begin{align*}
\rk(u\setminus a)
&\leq \rk(u) &&\text{by \eqref{x2}}\\
&\leq \rk(x)+\rk(y)-\rk(x\wedge y) &&\text{by \eqref{x3}}\\
&=|x\setminus a| + |y\setminus a| - (|(x\wedge y)\setminus
a|+1) && \text{since } x,y \text{ are minimally dependent}\\
&=|u\setminus a|-1 && \text{since the Boolean lattice } S_{\leq u}
\text{ is modular}\\
\end{align*}
which implies $u\setminus a$ is dependent. Therefore, there is some $z\in C(\M)$
with $z\leq u\setminus a$, as desired.

Conversely, suppose that $C$ satisfies \eqref{c1}--\eqref{c4}, and let
$I=\{x\in S\st x\not\geq c \text{ for any } c\in C\}$.
Then \eqref{c1}--\eqref{c3} imply that for every $u\in\max S$,
$C_{\leq u}$ is the collection of
circuits for a matroid on $S_{\leq u}$ with independence poset $I_{\leq
u}$. Since $I_{\leq u}$ satisfies \eqref{i1}--\eqref{i3} for all
$u\in\max S$, and they are local in nature, $I$ itself satisfies
\eqref{i1}--\eqref{i3}. 
Finally, \eqref{i4} follows directly from \eqref{c4} since 
$I_{\leq x} = \{z'\st z'\leq x, z'\not\geq \text{ for any } c\in C\}$.
Thus, $I$ is the independence poset for a matroid prescheme on $S$.
Noting that $C$ is an antichain by \eqref{c2}, we must have 
 $C=\min(S\setminus I)$ is the set of
circuits for this matroid prescheme.
\end{proof}

\section{Loops and isthmuses}\label{sec:loops}
In this section, we discuss the ideas of loops and isthmuses. In matroid theory,
an isthmus is sometimes called a coloop because it is a loop in the dual
matroid. In our context, we do not know how duality fits into the story of
matroid schemes and hence favor the term isthmus.

\begin{proposition}\label{prop:loop}
Let $\M=(S,\rk)$ be a matroid prescheme and $a\in \at(S)$ an atom. The following
are equivalent.
\begin{enumerate}
\item $\rk(a)=0$.
\item $a\leq x$ for any flat $x\in F(\M)$.
\item $a\not\leq b$ for any basis $b\in B(\M)$.
\end{enumerate}
\end{proposition}
\begin{proof}
\
\begin{description}
\item[$(1)\implies(2)$]
Let $x\in F(\M)$. Since $\rk(a)=0$, any $\ell\in x\wedge a$ must satisfy
$\rk(\ell)=\rk(a)$. Then $x\vee a\neq\emptyset$ by \eqref{x4}, so let $u\in
x\vee a$. We have $\rk(x)\leq\rk(u)\leq\rk(x)+\rk(a)-\rk(\ell)=\rk(x)$ by
\eqref{x2} and \eqref{x3}. It follows by \Cref{lem:closure} that $a\leq\cl(x)=x$.

\item[$(2)\implies(1)$]
We have by assumption that $a\leq \cl(\zero)$ and hence
$\rk(a)=\rk(\zero)=0$ by \Cref{lem:closure}.

\item[$(1)\implies(3)$]
Let $b\in B(\M)$. Since $\rk(a)<|a|$, we must have $a\notin I(\M)$, and so
by \eqref{i2} we have $a\not\leq b$.

\item[$(3)\implies(1)$]
We have by \eqref{x1} that $0\leq \rk(a)\leq |a|=1$. If $\rk(a)=1$, then $a\in
I(\M)$ and hence $a\leq b$ for some $b\in B(\M)=\max I(\M)$, which contradicts
our assumption. Thus, $\rk(a)=0$.
\end{description}
\end{proof}

\begin{proposition}\label{prop:isthmus}
Let $\M=(S,\rk)$ be a matroid prescheme and $a\in\at(S)$ an atom. The
following are equivalent.
\begin{enumerate}
\item if $x\in S_{\not\geq a}$ then $x\vee a\neq\emptyset$ and $\rk(u)=\rk(x)+1$
for any $u\in x\vee a$.
\item $a\leq b$ for any basis $b\in B(\M)$.
\item $a\leq m$ for any $m\in\max S$, and $a\not\leq c$ for any 
circuit $c\in C(\M)$.
\end{enumerate}
\end{proposition}
\begin{proof}
\
\begin{description}
\item[$(1)\implies(2)$]
Let $b\in B(\M)$. If $a\not\leq b$, then by assumption there is some
$u\in b\vee a$ with $\rk(u)=\rk(b)+1 = |b|+1=|u|$. But then $u\in I(\M)$ and
$u>b$, contradicting maximality of the basis $b$.

\item[$(2)\implies(3)$]
Let $m\in\max S$, and $y\in\max I(\M)_{\leq m}$ so that $\rk(y)=\rk(m)$.
If $y$ is not a basis, then there is a basis $b>y$ and \eqref{x4} would
contradict maximality of $m$. So $y$ must be a basis, 
thus by assumption $a\leq y\leq m$.

Now let $c\in C(\M)$, and assume for contradiction that $a\leq c$. By minimality
of $c$, we have $c\setminus a\in I(\M)$. Let $b\in B(\M)$ with $b\geq c\setminus
a$. Then by assumption $a\leq b$, so there exists some $v\in (c\setminus a)\vee
a$ with $v\leq b$. By \eqref{i2},
 we have $\rk(v)=|v|$. 
Then $c\in (c\setminus a)\vee a$ 
implies $\rk(c)=\rk(v)=|v|=|c|$, contradicting that $c$ is dependent.

\item[$(3)\implies(1)$]
Let $x\in S_{\not\geq a}$. Since $a\leq m$ for any $m\in \max S$, we must have
$x\vee a\neq\emptyset$, so fix $u\in x\vee a$.
Also let $b\in \max
I(\M)_{\leq x}$ and $v\in b\vee a$ with $v\leq u$. If $\rk(v)=\rk(b)$, then there
exists a $c\in C(\M)$ with $c\leq v$. But since $b$ is independent, this implies
$c\geq a$, a contradiction. Thus, 
$\rk(x)=\rk(b)<\rk(v)\leq \rk(u)\leq \rk(x)+1$
and therefore $\rk(u)=\rk(x)+1$.
\end{description}
\end{proof}

\begin{definition}\label{def:loopisthmus}
In a matroid prescheme $\M=(S,\rk)$, we say an atom $a\in \at(S)$ is a
\textbf{loop} if any of the conditions of \Cref{prop:loop} hold, and $a$
is an \textbf{isthmus} if any of the conditions of
\Cref{prop:isthmus} hold.
\end{definition}

\begin{remark}
It follows that for a matroid prescheme $\M=(S,\rk)$ and an atom $a\in \at(S)$:
\begin{itemize}
\item $a$ is a loop of $\M$ if
and only if $a$ is a loop of some local matroid $\M_u$ where $u\in\max
S$.
\item $a$ is an isthmus of $\M$ if and only if $a$ is an isthmus of
every local matroid $\M_u$ where $u\in\max S$.
\end{itemize}
\end{remark}

\section{Deletion and contraction}
In this section, we discuss the deletion and contraction operations on
matroid preschemes. 
It is straightforward to check that each construction is in fact a
matroid prescheme. However, contraction might not preserve the global
property \eqref{x5} (see \Cref{ex:contr,fig:delcontr}).

\begin{proposition}\label{prop:del}
Let $\M=(S,\rk)$ be a matroid prescheme and $a\in\at(S)$ an atom. 
Then
$\M-a:=(S_{\not\geq a},\rk|_{S_{\not\geq a}})$ is a matroid prescheme.
If $\M$ is a matroid scheme, then $\M-a$ is a matroid scheme with rank
equal to $\rk(\M)-1$ if $a$ is an isthmus and $\rk(\M)$ otherwise.
\end{proposition}
\begin{proof}
This is straightforward to verify.
\end{proof}

\begin{definition}\label{def:del}
Let $\M=(S,\rk)$ be a matroid prescheme and $a\in\at(S)$ an atom. The
\textbf{deletion} of $\M$ by $a$ is the matroid
prescheme $\M-a=(S_{\not\geq a},\rk|_{S_{\not\geq a}})$ from \Cref{prop:del}.
\end{definition}

\begin{remark}\label{rmk:restrict}
One can also define the \textbf{restriction} of a matroid prescheme by repeatedly
applying the deletion operation.
Given a set of atoms $A\subseteq\at(\S)$, the successive deletion of atoms not
in $A$ gives rise to a matroid prescheme $\M[A]:=(\angles{A},\rk|_{\angles{A}})$ where
$\rk$ is restricted to the order ideal
\[\angles{A}:=\{x\in \S \st x\in\bigvee T, \text{ for some } T\subseteq A\}.\]
\end{remark}

\begin{proposition}\label{prop:contr}
Let $\M=(S,\rk)$ be a matroid prescheme and $x\in S$. Define $\rk_{/x}:S_{\geq
x}\to\Z_{\geq0}$ by $\rk_{/x}(w)=\rk(w)-\rk(x)$. 
Then
 $\M_{/x}:=(S_{\geq x},\rk_{/x})$ is a matroid  prescheme.
If $\M$ is pure of rank $\rk(\M)$, then $\M_{/x}$ is pure with
rank equal to $\rk(\M)-\rk(x)$.
\end{proposition}
\begin{proof}
This is straightforward to verify.
\end{proof}

\begin{definition}\label{def:contr}
Let $\M=(S,\rk)$ be a matroid prescheme and $x\in S$. The
\textbf{contraction} of $\M$ by $x$ is the matroid prescheme $\M_{/x}=(S_{\geq
x},\rk_{/x})$ from \Cref{prop:contr}.
\end{definition}

\begin{example}\label{ex:contr}
In general, the contraction of a matroid scheme might not satisfy
condition \eqref{x5}. For example, consider
the matroid scheme $\M=(\S,\rk)$ where $S$ is the simplicial poset
depicted in \Cref{fig:M} and $\rk(w)=|w|$ for all $w\in
\S$. A contraction $\M_{/a}$ by atom $a$ is depicted in
\Cref{fig:M/a} and does not satisfy \eqref{x5}.

\begin{figure}[hb]
\begin{subfigure}[t]{.3\textwidth}
\centering
\begin{tikzpicture}[scale=1.3]
\node at (-.2,.8) {$a$};
\labelnodes
\node (0) at (0,0) {$0$};
\node (u) at (-1,3) {$3$};
\node (v) at (1,3) {$3$};
\foreach \x in {1,2,3} {
\node (a\x) at (\x-2,1) {$1$};
\node (b\x) at (0.5*\x-2,2) {$2$};
\node (c\x) at (0.5*\x,2) {$2$};
\draw[-] (a\x)--(0);
\draw[-] (b\x)--(u);
\draw[-] (c\x)--(v);
}
\foreach \x in {1,2} {
\draw[-] (c\x)--(a1)--(b\x);
}
\foreach \x in {1,3} {
\draw[-] (c\x)--(a2)--(b\x);
}
\foreach \x in {3,2} {
\draw[-] (c\x)--(a3)--(b\x);
}
\end{tikzpicture}
\caption{$\M$}
\label{fig:M}
\end{subfigure}
\hfill
\begin{subfigure}[t]{.3\textwidth}
\centering
\begin{tikzpicture}[scale=1.3]
\opennodes
\node (0) at (0,0) {};
\node (b2) at (-1,2) {};
\node (c2) at (1,2) {};
\labelnodes
\node (u) at (-1,3) {$2$};
\node (v) at (1,3) {$2$};
\node (a2) at (0,1) {$0$};
\foreach \x in {1,3} {
\opennodes
\node (a\x) at (\x-2,1) {};
\labelnodes
\node (b\x) at (0.5*\x-2,2) {$1$};
\node (c\x) at (0.5*\x,2) {$1$};
\draw[-,dashed] (a\x)--(0);
\draw[-,thick] (b\x)--(u);
\draw[-,thick] (c\x)--(v);
}
\foreach \x in {1,2} {
\draw[-,dashed] (c\x)--(a1)--(b\x);
}
\foreach \x in {1,3} {
\draw[-,thick] (c\x)--(a2)--(b\x);
}
\foreach \x in {3,2} {
\draw[-,dashed] (c\x)--(a3)--(b\x);
}
\draw[-,dashed] (0)--(a2);
\draw[-,dashed] (b2)--(u);
\draw[-,dashed] (c2)--(v);
\end{tikzpicture}
\caption{$\M_{/a}$}
\label{fig:M/a}
\end{subfigure}
\hfill
\begin{subfigure}[t]{.3\textwidth}
\centering
\begin{tikzpicture}[scale=1.3]
\labelnodes
\node (0) at (0,0) {$0$};
\node (b2) at (-1,2) {$2$};
\node (c2) at (1,2) {$2$};
\opennodes
\node (u) at (-1,3) {};
\node (v) at (1,3) {};
\node (a2) at (0,1) {};
\foreach \x in {1,3} {
\labelnodes
\node (a\x) at (\x-2,1) {$1$};
\opennodes
\node (b\x) at (0.5*\x-2,2) {};
\node (c\x) at (0.5*\x,2) {};
\draw[-,thick] (a\x)--(0);
\draw[-,dashed] (b\x)--(u);
\draw[-,dashed] (c\x)--(v);
}
\draw[-,dashed] (b1)--(a1)--(c1);
\draw[-,thick] (a3)--(b2)--(a1)--(c2)--(a3);
\draw[-,dashed] (b3)--(a3)--(c3);
\foreach \x in {1,3} {
\draw[-,dashed] (c\x)--(a2)--(b\x);
}
\draw[-,dashed] (0)--(a2);
\draw[-,dashed] (b2)--(u);
\draw[-,dashed] (c2)--(v);
\end{tikzpicture}
\caption{$\M-a$}
\label{fig:M-a}
\end{subfigure}

\caption{A matroid scheme $\M$ alongside its contraction by atom $a$ and
deletion of atom $a$. See \Cref{ex:contr}.}
\label{fig:delcontr}
\end{figure}

In order for all contractions of a matroid scheme $\M$ to again be
matroid schemes, we would need $\M=(\S,\rk)$ to satisfy a stronger
condition such as:
\begin{enumerate}
\item[\mylabel{m5''}{\textbf{M5''}}] 
if $x,y\in S$, $\rk(x)<\rk(y)$, and $\ell\in x\wedge y$, then there is
some $z\in S$ with $\ell<z\leq y$ and $z\vee x\neq\emptyset$.
\end{enumerate}

\end{example}

\begin{remark}
When $\M=(S,\rk)$ is a semimatroid and $a\in \at(S)$,
Ardila's definition of contraction for by $a$ \cite[Def.
7.4]{ardila} would translate to the matroid scheme on the
simplicial semilattice 
\[S_{/a}:=\{z\in S\st z\not\geq a, z\vee a\neq\emptyset\}\]
with the rank of $z$ equal to $\rk(z\vee a)-\rk(a)$.
This agrees with our definition via the poset isomorphism
$\phi:S_{\geq a}\to S_{/a}$ which takes $u\in S_{\geq a}$ to $u\setminus a$, the
complement of $a$ in the Boolean lattice $S_{\leq u}$.

When $S$ is not a semilattice, the analogous map $\phi$ is generally surjective
but not injective.
We chose our formulation of contraction because, when the matroid scheme is realized
by an abelian arrangement, the contraction represents a topological operation
in the same way matroid contraction does for hyperplane arrangements (see
\Cref{thm:arr}).
\end{remark}

\begin{proposition}\label{prop:del=contr}
Let $\M=(S,\rk)$ be a matroid prescheme and $a\in \at(S)$. If $a$ is a loop,
then there is an isomorphism $\M-a\cong\M_{/a}$.
\end{proposition}
\begin{proof}
We prove that there is a poset isomorphism $\phi:S_{\geq a}\to S_{\not\geq a}$
for which $\rk(z)=\rk(\phi(z))$.
The map $\phi$ is defined by setting $\phi(z)=z\setminus a$ where $z\setminus a$
is the unique complement of $a$ in the Boolean lattice $S_{\leq z}$.
It is clear that this map is order-preserving, and it is rank-preserving because
\eqref{x2}, \eqref{x3}, and $\rk(a)=0$ together imply
\[\rk(z)=\rk(z)+\rk(\zero)\leq \rk(z\setminus a)+\rk(a) = \rk(z\setminus a)\leq
\rk(z)\]
thus $\rk(z)=\rk(z\setminus a)$.

It now suffices to prove that for any $u\in S_{\not\geq a}$, there exists a
unique $z\in u\vee a$.
Such a $z$ must exist by \eqref{x4}, since $\zero=u\wedge a$ and
$\rk(\zero)=\rk(a)$.
For uniqueness, suppose that $z,z'\in u\vee a$ and $\ell\in z\wedge z'$ such
that $u\leq \ell$. Then \eqref{x2}, \eqref{x3}, and $\rk(a)=0$ together imply
\[\rk(z)=\rk(z)+\rk(\zero)\leq \rk(u)+\rk(a)=\rk(u)\leq \rk(\ell)\leq \rk(z),\]
thus $\rk(\ell)=\rk(z)$. Then \eqref{x4} implies that $z$ and $z'$ have a
common upper bound, say $m$. But then $u\vee a$ must be unique in the lattice
$S_{\leq m}$, hence  $z=z'$.
\end{proof}

\begin{remark}\label{rmk:isthmus}
If $\M$ is a semimatroid, then it is also true
that $\M-a\cong\M_{/a}$ when $a$ is an isthmus \cite[Prop. 7.12]{ardila}.
However, the same does not hold in our generality. The added difficulties we have
surrounding an isthmus is likely correlated to the difficulty of matroid duality
in our setting.
As an example, consider the matroid scheme depicted in
\Cref{fig:delcontr} along with
its contraction and deletion with respect to the isthmus $a$.
\end{remark}

\section{Geometric posets}
In this section, we classify posets which arise as the poset of flats of a
matroid (pre)scheme. This generalizes the well-known result that a
simple matroid is equivalent to a geometric lattice.

\begin{definition}\label{def:geom}
A \textbf{locally geometric poset} is a bounded below, ranked poset
$(P,\rk)$ satisfying:
\begin{enumerate}
\item[\mylabel{g1}{\textbf{G1}}] 
Every maximal interval is a geometric lattice.
\end{enumerate}
A \textbf{geometric poset} is a locally geometric poset satisfying the
additional condition:
\begin{enumerate}
\item[\mylabel{g2}{\textbf{G2}}]
for every $x\in P$, $A\subseteq \at(P)$, and $y\in \bigvee A$ such that
$\rk(x)<\rk(y)=|A|$, there exists an $a\in A$ such that $a\not\leq x$ and $a\vee
x\neq\emptyset$.
\end{enumerate}
\end{definition}

\begin{remark}
A lattice or a meet-semilattice is geometric in the sense of \Cref{def:geom} if
and only if it is a geometric lattice or geometric semilattice. Indeed, it is
easy to see that axioms \eqref{g1} and \eqref{g2} translate exactly to the
criteria (G3) and (G4) of Wachs--Walker \cite[Thm 2.1]{WW}.
This definition of geometric poset was also given in \cite[Def. 4.1.1]{BD}, but
here our goal is to make the connection to matroid schemes.
\end{remark}

\begin{example}\label{ex:geom}
The posets depicted in \Cref{fig:flats} are locally geometric.
The first two of those posets also satsify \eqref{g2}, hence are
geometric posets. The third poset in \Cref{fig:flats3} is not
geometric: the elements labelled by $x$ and $y$ do not satisfy the
requirement of $\eqref{g2}$.

\end{example}

Geometric posets arise as the poset of flats of a matroid scheme, however there
are many matroid schemes which have isomorphic posets of flats. Nevertheless,
there is a particular representative matroid scheme which we can single out using
the following definition.
\begin{definition}\label{def:simple}
A matroid scheme $\M=(S,\rk)$ is \textbf{simple} if every atom $a\in \at(S)$ is a flat
and not a loop.
\end{definition}

\begin{example}
The matroid scheme in \Cref{fig:ex1} is simple while that in
\Cref{fig:ex2} is not (since the atoms are not flats).
 In fact, for this
non-simple matroid scheme, its poset of flats (\Cref{fig:flats2})
can itself be viewed as a matroid scheme (with the same poset of flats).
\end{example}

\begin{theorem}\label{thm:GP=MS}
A poset is geometric if and only if it is isomorphic to the poset of flats of a
matroid scheme. Furthermore, each geometric poset is the poset of flats of a unique
simple matroid scheme, up to isomorphism.
The same is true for locally geometric posets and matroid preschemes.
\end{theorem}
\begin{proof}
Let $\M=(\S,\rk)$ be a matroid scheme and $\F=\F(\M)$ its poset of flats; we will
show that $\F$ is geometric. We immediately have that  $\F$ is bounded below by
$\cl(\zero)$, and it is ranked and locally geometric 
since for any $x\in\F$, $\F_{\leq x}$ is the lattice of flats of the local
matroid $\M_x=(S_{\leq x},\rk|_{S_{\leq x}})$ (\Cref{prop:localflats}).

To see that $\F$ satisfies \eqref{g2}, let $x\in\F$, $A\subseteq \at(\F)$,
and $y\in\bigvee_F A$ such that $\rk(x)<\rk(y)=|A|$. 
We may pick, for each $a\in A$, some 
$\tilde{a}\in\at(S_{\leq y})$ with $\cl(\tilde{a})=a$, such that there is some
$z\in\bigvee_{a\in A}\tilde{a}$ with $\cl(z)=y$ (\Cref{lem:closure2}).
Then $\rk(x)<\rk(z)$, and so by \eqref{x5} there is some $a\in A$ such that
$\tilde{a}\leq_S z\leq_S y$, $\tilde{a}\not\leq_S x$, and $x\vee_S
\tilde{a}\neq\emptyset$.
Now, $\tilde{a}\not\leq_S x$ and $\tilde{a}\leq_S a$ implies $a\not\leq_S x$,
hence $a\not\leq_F x$.
Moreover, for any $u\in\tilde{a}\vee_S x$, $\cl(u)$ is an upper bound of both $a$
and $x$ (in both $S$ and $F$), hence $a\vee_F x\neq\emptyset$.
This implies that $\F(\M)$ is a geometric poset.

For the converse, 
 let $P$ be a geometric poset with rank function $\rank_P$; we will construct
a matroid scheme whose poset of flats is equal to $P$.
Define a poset $\S$ as the set \[S = \{(I,x)\in
\boo(\at(P))\times P\st x\in\bigvee_P I\}\] with partial order \[(I,x)\leq_S(J,y) \iff 
I\subseteq J \text{ and } x\leq_P y.\] 
To see that $\S$ is a simplicial poset, let $(I,x)\in\S$, and we claim that
the function $S_{\leq(I,x)}\to\boo(I)$ given by $(J,y)\mapsto J$ is an
isomorphism of posets. Indeed, for any $J\subseteq I$, $x$ is an upper bound of
$J$ in $P$ and hence there is a $y\in\bigvee_P J$ such that $y\leq_P x$; $y$ is
unique because $P_{\leq x}$ is a lattice.

On this simplicial poset, define a function $\rk:\S\to\Z_{\geq0}$ by $\rk(I,x)
= \rank_P(x)$.
Using the fact that $P$ is locally geometric, we see that the local conditions
\eqref{x1}--\eqref{x3} are immediate. Indeed, for
any $(I,x)\in\S$, restricting $\rk$ to $S_{\leq(I,x)}$
defines a matroid structure whose lattice of flats is $P_{\leq x}$.

For \eqref{x4},
suppose that $(L,\ell)\in(I,x)\wedge_S(J,y)$ such that
$\rank_P(\ell)=\rk(L,\ell)=\rk(I,x)=\rank_P(x)$. 
Since $\ell\leq_P x$ and $\rank_P(\ell)=\rank_P(x)$, we must have
$\ell=x\in\bigvee_P I$. Since $y\geq_P \ell$ and $y\in\bigvee_P J$, we have that
$y\geq_P a$ for all $a\in I\cup J$. It follows that $(I\cup J,y)\in
(I,x)\vee_S(J,y)$.

To see \eqref{x5}, suppose that $\rank_P(x) = \rk(I,x) < \rk(J,y) = \rank_P(y)$.
Then \eqref{g2} implies there is some $a\in J$ such that $a\leq_P y$,
$a\not\leq_P x$, and $x\vee_P a\neq\emptyset$. Letting $u\in x\vee_P a 
\subseteq \bigvee_P (I\cup \{a\})$, we
have $(\{a\},a)\leq_S (J,y)$, $(\{a\},a)\not\leq_S (I,x)$, and $(I\cup \{a\},u)\in
(\{a\},a)\vee_S(I,x)$.

We need to show that $P$ is the poset of flats of the matroid scheme $(S,\rk)$ just
constructed. Via the injection $P\into S$ given by $x\mapsto (\at(P)_{\leq
x},x)$, it suffices to show that for any $(I,x)\in S$, $\cl(I,x) = (\at(P)_{\leq
x},x)$. If $(I,x)\leq (J,y)$, then $x\neq y$ implies $x<y$ and hence
$\rk(I,x) = \rank_P(x) < \rank_P(y) = \rk(J,y)$.
So $(I,x)\leq (J,y)$ and $\rk(I,x) = \rk(J,y)$ implies $x=y$ and $I\subseteq J$.
This means that $\cl(I,x) = (A,x)$ where $A$ is maximal in $\boo(\at(P))$ such
that $x\in\bigvee A$, which is precisely the set $\at(P)_{\leq x}$.

Therefore, $(S,\rk)$ is a matroid scheme whose poset of flats is isomorphic to $P$.
It is easy to see from construction that this matroid scheme is simple.
It remains to show that this is the only such simple matroid scheme, up to
isomorphism.

Suppose that $\M=(S,\rk)$ and $\M'=(S',\rk')$ are simple matroid schemes and
$\phi:\F(\M)\to\F(\M')$ is an isomorphism between their posets of flats.
Since $\M$ and $\M'$ are both simple, $\phi$ restricts to a bijection on atoms
$\at(S)\to \at(S')$. We will extend $\phi$ to a map $\psi:S\to S'$ as follows.
Let $x\in S$. Then for any $a\in \at(S_{\leq x})$, we have $a=\cl(a)\leq_S \cl(x)$
so $\phi(a)\leq_{S'} \phi(\cl(x))$. This means that there is a unique
$\psi(x)\in\bigvee_{S'}\{\phi(a)\st a\in\at(S_{\leq x})\}$ such that
$\psi(x)\leq_{S'}
\phi(\cl(x))$. 

We claim that this $\psi:S\to S'$ defines an isomorphism of
posets with $\rk(x)=\rk'(\psi(x))$.
If $x\leq_S y$, then $\cl(x)\leq_{S} \cl(y)$ so $\psi(x)\leq_{S'} \phi(\cl(y))$.
Since $\at(S_{\leq x})\subseteq \at(S_{\leq y})$ and $S'_{\leq \phi(\cl(y))}$ is
a lattice, we have $\psi(x)\leq_{S'}\psi(y)$. Thus, $\psi$ is order-preserving.
Moreover, one can use $\phi^{-1}$ to analogously define an order-preserving
inverse of $\psi$.
$\cl'(\psi(x))=\phi(\cl(x))$, hence
$\rk(x)=\rk(\cl(x))=\rk'(\phi(\cl(x)))=\rk'(\psi(x))$.
\end{proof}

\section{Tutte polynomial}
The Tutte polynomial, originally defined for a graph and later extended to
matroids, is well-known for having a deletion-contraction recurrence. 
In our setting, since contraction does not preserve \eqref{x5} in
general, but does preserve purity, we have to extend the notion of Tutte
polynomial to pure matroid preschemes.

\begin{definition}\label{def:tutte}
Given a pure matroid prescheme $\M=(\S,\rk)$, define the Tutte polynomial as
\[T_{\M}(\x,\y) = \sum_{w\in\S} 
(\x-1)^{\rk(\M)-\rk(w)}(\y-1)^{|w|-\rk(w)}.\]
\end{definition}
If $\M$ is a matroid or a semimatroid, this coincides with the usual
definition of the Tutte polynomial.

\begin{theorem}\label{thm:tutte}
Let $\M=(S,\rk)$ be a matroid scheme and $a\in \at(S)$ an atom. 
\begin{enumerate}
\item If $a$ is not a loop nor an isthmus, then 
\[T_{\M}(\x,\y) = T_{\M-a}(\x,\y) + T_{\M/a}(\x,\y).\]
\item If $a$ is a loop, then \[T_{\M}(\x,\y) = \y T_{\M/a}(\x,\y).\]
\item If $a$ is an isthmus, then 
\[T_{\M}(\x,\y) = (\x-1)T_{\M-a}(\x,\y) + T_{\M/a}(\x,\y).\]
\end{enumerate}
\end{theorem}
\begin{proof}
To start, we have
\begin{align*}
T_{\M}(\x,\y) 
&= \sum_{w\in S} (\x-1)^{\rk(\M)-\rk(w)}(\y-1)^{|w|-\rk(w)}\\
&= \sum_{w\in S_{\not\geq a}} (\x-1)^{\rk(\M)-\rk(w)}(\y-1)^{|w|-\rk(w)}
+ \sum_{w\in S_{\geq a}} (\x-1)^{\rk(\M)-\rk(w)}(\y-1)^{|w|-\rk(w)}\\
&= (\x-1)^{\rk(\M)-\rk(\M-a)} T_{\M-a}(\x,\y)
+ (\y-1)^{|a|-\rk(a)} T_{\M/a}(\x,\y),
\end{align*}
Recall from \Cref{prop:del} that $\rk(\M-a)=\rk(\M)$ if $a$ is not an isthmus,
and $\rk(\M-a)=\rk(\M)-1$ if $a$ is an isthmus. 
We use this to simplify the above expression in each of our three cases.

If $a$ is neither a loop nor an isthmus, then $\rk(\M)-\rk(\M-a)=0$
and $|a|-\rk(a)=0$, so the expression simplifies to
$T_{\M-a}(\x,\y)+T_{\M/a}(\x,\y)$, as desired.

If $a$ is a loop, then $\rk(\M)-\rk(\M-a)=0$ and $|a|-\rk(a)=1$. Then the above
expression simplifies to $T_{\M-a}(\x,\y)+(\y-1) T_{\M/a}(\x,\y)$, and our
desired formula follows from the isomorphism $\M-a\cong \M/a$ from
\Cref{prop:del=contr}.

If $a$ is an isthmus, then $\rk(\M)-\rk(\M-a)=1$ and $|a|-\rk(a)=0$. Then the
above expression simplifies to $(\x-1)T_{\M-a}(\x,\y)+T_{\M/a}(\x,\y)$, as
desired.

\end{proof}

\begin{example}
Consider the matroid scheme $\M$ depicted in \Cref{fig:delcontr}, along
with its
deletion and contraction by an isthmus $a$.
The recurrence of \Cref{thm:tutte} yields
\begin{align*}
T_{\M}(\x,\y) &= (\x-1)T_{\M-a}(\x,\y) + T_{\M/a}(\x,\y)\\
&=(\x-1)(\x^2+1) + (\x^2+2\x-1)\\
&=\x^3+3\x-2
\end{align*}
We  see here that, unlike in the classical setting, the coefficients of
our Tutte polynomial need not be nonnegative. Nevertheless, the Tutte
polynomial does encode some invariants, which we explore next.
\end{example}

\begin{proposition}
Let $\M=(S,\rk)$ be a pure matroid prescheme. Then $T_{\M}(1,1)=|B(\M)|$
and $T_{\M}(2,2)=|S|$.
\end{proposition}
\begin{proof}
Setting $\x=\y=1$, the only summands in the definition of $T_{\M}(\x,\y)$
are indexed by $w\in S$ with $\rk(\M)=\rk(w)=|w|$. These elements are precisely
the bases of $\M$, and the corresponding summand is 1.
Setting $\x=\y=2$, all summands in the definition of $T_{\M}(\x,\y)$ are
equal to 1.
\end{proof}

\section{Characteristic polynomial}
Recall for a poset $P$ bounded below by $\zero$, we can recursively
define a M\"obius function
$\mu_P:P\to\Z$ so that $\mu_P(\zero)=1$ and $\sum_{u\leq
w}\mu_P(u)=0$ for any $w\in P_{>\zero}$. 
Then when $(P,\rk)$ is ranked and pure of rank $\rk(P)$, we define the
characteristic polynomial by
\[\chi_P(t) := \sum_{w\in P} \mu_P(w) t^{\rk(P)-\rk(w)}.\]
There is a well-known relationship between the Tutte polynomial of a matroid and
the characteristic polynomial of its poset of flats, which we extend to matroid
schemes in \Cref{thm:charpoly}.

\begin{lemma}\label{lem:moebius}
Let $\M=(S,\rk)$ be a matroid prescheme with no loops, and let $\F=\F(\M)$ be its poset
of flats. Then for any $w\in\F$,
\[\mu_{\F}(w) = \sum_{\substack{u\in S\st\\\cl(u)=w}} (-1)^{|u|}.\]
\end{lemma}
\begin{proof}
Since $\M$ has no loops, the minimum element of $\F$ is $\cl(\zero)=\zero$ and
$\mu_{\F}(\zero) = 1 = (-1)^{|\zero|}$.
Now let $w\in\F$ with $w>\zero$. Because $S_{\leq w}$ is a Boolean lattice of
positive rank, 
\[\sum_{u\in S_{\leq w}} (-1)^{|u|} = 0.\]
Using this, we have by induction:
\[\mu_{\F}(w) = - \sum_{\substack{v\in\F\st\\v<w}} \mu_{\F}(v)
= - \sum_{\substack{v\in\F\st\\v<w}} \sum_{\substack{u\in S\st\\\cl(u)=v}}
(-1)^{|u|}
= - \sum_{\substack{u\in S\st\\ \cl(u)<w}} (-1)^{|u|} 
= \sum_{\substack{u\in S\st\\\cl(u)=w}} (-1)^{|u|}\]
\end{proof}

\begin{theorem}\label{thm:charpoly}
If $\M=(S,\rk)$ is a pure matroid prescheme with no loops, then the
characteristic 
polynomial of its poset of flats $\F=\F(\M)$ is
\[\chi_{\F}(t) = (-1)^{\rk(\M)} T_{\M}(1-t,0).\]
\end{theorem}
\begin{proof}
\begin{align*}
(-1)^{\rk(\M)} T_{\M}(1-t,0)
&= (-1)^{\rk(\M)} \sum_{u\in S} (-t)^{\rk(\M)-\rk(u)}(-1)^{|u|-\rk(u)}\\
&= \sum_{u\in S} (-1)^{|u|} t^{\rk(\M)-\rk(u)}\\
&= \sum_{w\in \F}\sum_{\substack{u\in S\st\\\cl(u)=w}} (-1)^{|u|}
t^{\rk(\M)-\rk(u)}\\
&= \sum_{w\in \F} \mu_{\F}(w)t^{\rk(\M)-\rk(w)}=\chi_F(t) &&\text{by Lemma
\ref{lem:moebius}}
\end{align*}
\end{proof}

\section{Semimatroids with group actions}
\Cref{def:semi} of a semimatroid
 extends directly to allow
simplicial complexes which are finite-dimensional but with an infinite vertex
set. In this section, we will allow this generality and consider group actions
on a semimatroid following \cite{DR,dd}. 

An action of a group $G$ on
a semimatroid $\M=(\cx,\rk)$ is an action of $G$ on the vertex set $\at(\cx)$
whose induced action on $\boo(\at(\cx))$ preserves rank and centrality. That is,
for any $g\in G$ and $A\subseteq\at(\cx)$, one has $A\in\cx$ if and only if
$gA\in\cx$, and in this case $\rk(A)=\rk(gA)$.
We say that the action of $G$ is \textbf{cofinite} if there are finitely many
$G$-orbits, and \textbf{translative} if for any $g\in G$ and $a\in\at(\cx)$,
$\{a,ga\}\in\cx$ implies $ga=a$.

Given a group $G$ acting on a ranked poset $P$, the quotient $P/G$ is defined as
the set of orbits $\{Gx \st x\in P\}$ with partial order $Gx\leq Gy$ if $gx\leq
y$ for some $g\in G$. Since $P$ is ranked, the quotient $P/G$ will indeed be a
poset (see \cite[Lemma 2.6]{dd}). Moreover, if $P$ is a simplicial poset then
$P/G$ will be a simplicial poset if and only if the action of $G$ is translative
(see \cite[Lemma 3.5]{dd}).

Consider a cofinite translative action of a group $G$ on a finite-dimensional
simplicial complex $\cx$.
For $A\subseteq \at(\cx)/G$, define $m_G(A)$ to be the number of $G$-orbits
under the action of $G$ on the set of all $\{a_1,\dots,a_k\}\in\cx$ for which
$\{Ga_1,\dots,Ga_k\}=A$.
If moreover $\rk:\cx\to\Z_{\geq0}$ is a $G$-invariant function, \cite[Def.
3.28]{DR} define an analogue of the Tutte polynomial for the $G$-semimatroid
$\M=(\cx,\rk)$ to be
\[T_{G\actson\M}(\x,\y) := \sum_{A\subseteq\at(\cx)/G}
m_G(A)(\x-1)^{\rk(\M)-\rk(A)}(\y-1)^{|A|-\rk(A)}.\]

\begin{theorem}\label{thm:Gsemi}
Let $\M=(\cx,\rk)$ be a semimatroid with a translative and cofinite action of a
group $G$, and 
define $\rk_G:\cx/G\to\Z_{\geq 0}$ by $\rk_G(Gx)=\rk(x)$. 
Then $\M/G=(\cx/G,\rk_G)$ is a matroid scheme with 
\[ F(\M/G)=F(\M); \quad I(\M/G) = I(\M)/G; \quad C(\M/\G) = C(\M)/\G; \quad 
\text{and} \quad 
T_{\M/G}(\x,\y) = T_{G\actson \M}(\x,\y).\]
\end{theorem}
\begin{proof}
Axioms \eqref{x1}--\eqref{x4} only require that $\rk$ is $G$-invariant;
the translativity hypothesis is required to see that $\cx/G$ is a simplicial
poset (which follows from \cite[Lemma 3.5]{dd}) and to establish \eqref{x5}.
\begin{description}
\item[\eqref{x1}]
For any $x\in\cx$, 
if
$\rk(x)\leq \rank_{\cx}(x)$.
then
$\rkG(\G x)\leq \rank_{\cx/\G}(\G x)$.
\item[\eqref{x2}]
If $Gx\leq Gy$, then $x\leq gy$ for some $g\in G$, thus
$\rkG(\G x) = \rk(x) \leq \rk(gy) = \rkG(\G y)$.
\item[\eqref{x3}]
If $\G u\in\G x\vee\G y$, then there exist $g,h\in\G$ such that $u\in gx\vee
hy$, and hence
\begin{align*}
\rk(u)+\rk((gx)\wedge(hy))
&=\rkG(\G u) + \rkG(\G x\wedge\G y)\\
&\leq \rkG(\G x) + \rkG(\G y)
= \rk(gx)+\rk(hy)
\end{align*}

\item[\eqref{x4}]
Suppose that $G\ell\in(Gx)\wedge(Gy)$ and $\rk_G(G\ell)=\rk_G(Gx)$. Then there
exist $g,h\in G$ such that $\ell\in(gx)\wedge(hy)$ and $\rk(\ell)=\rk(gx)$. This
implies $(gx)\vee(hy)\neq\emptyset$ and hence $(Gx)\vee(Gy)\neq\emptyset$.
\item[\eqref{x5}]
If $\rkG(\G x)<\rkG(\G y)$, then $\rk(x)<\rk(y)$. So there is some $a\in A(\cx)$
such that $a\leq y$, $a\not\leq x$, and $a\vee x\neq\emptyset$. It follows that
$\G a\in A(\cx/\G)$, $\G a\leq\G y$, and $\G a\vee\G x\neq\emptyset$.
We must also have $\G a\not\leq\G x$, since otherwise
there would be some $h\in\G$ with $ha\leq x$,
contradicting translativity since this would make 
any $z\in a\vee x$ a common upper bound of
$a$ and $ha$. 
\end{description}
Further note that the closure commutes with the $\G$-action, implying that
$F(\M)/G= F(\M/\G)$. The statements about independence and circuits follow
from the assumption that $\rk$ is $G$-invariant.

For the Tutte polynomial, the key is to observe that $\boo(\at(\cx)/G)$ indexes
a partition of $\cx/G$ with blocks given by the sets $\bigvee A$ for
$A\subseteq\at(\cx)/G$ (allowing for some blocks to be empty). Then
$m_G(A)=|\bigvee A|$ and $\rk(A)=\rk_G(w)$ for any $w\in\bigvee A$. Thus
\[\sum_{w\in
\cx/G}(\x-1)^{\rk_G(\M/G)-\rk_G(w)}(\y-1)^{|w|-\rk_G(w)}
= 
\sum_{A\subseteq \at(\cx)/G} m_G(A)(\x-1)^{\rk(\M)-\rk(A)}(\y-1)^{|A|-\rk(A)}\]
\end{proof}

\begin{example}\label{ex:nonpos}
The matroid scheme $\M=(S,\rk)$ depicted in
\Cref{fig:M}, with $\rk(w)=|w|$ for all $w\in S$,
is not isomorphic to $\widetilde{\M}/G$ for any translative $G$-action on a
semimatroid $\widetilde{\M}=(\widetilde{S},\widetilde{\rk})$.
If it were, then letting $u,v\in\max\widetilde{S}$ so that $Gu\neq Gv$, repeatedly
applying the basis exchange axiom \eqref{b3'} in $\widetilde{\M}$ would yield a
path $Gu=Gw_0>Gz_1<Gw_1>Gz_2<Gw_2>Gz_3<Gw_3=Gv$ in $S$, where $\rk(Gz_i)=2$ and
$\rk(Gw_i)=3$ for each $i$, which does not exist.
Alternatively, its Tutte polynomial 
\[T_{\M}(\x,\y) = 2+6(\x-1)+3(\x-1)^2+(\x-1)^3 = \x^3+3\x-2.\]
contains a negative coefficient, which never happens for
quotients of semimatroids (\cite[Thm. H]{DR}).

\end{example}

\section{Abelian arrangements}\label{sec:arr}
The topological motivation for matroid schemes and geometric posets is through
certain arrangements of submanifolds, which we make explicit here.
An example is given in \Cref{ex:toric} (see also \Cref{fig:toric}).
Throughout this section, 
fix a finite-dimensional connected abelian Lie group $\GG$ and a finite-rank
free abelian group $\Gamma$.

\begin{definition}\label{def:arr}
An abelian Lie group arrangement, or an \textbf{abelian arrangement} for short,
is a collection $\A=\{H_{\alpha,t}\st(\alpha,t)\in\X\}$ for some finite set
$\X\subseteq\Gamma\times\GG$, where 
\[H_{\alpha,t} := \{\varphi\in\Hom(\Gamma,\GG)\st \varphi(\alpha)=t\}.\]
We may and will assume that the collection $\X$ is chosen so that each
$H_{\alpha,t}$ is connected.

\end{definition}

\begin{remark}
When $\GG=\R^d$ and $\Gamma=\Z^n$, this is an arrangement of affine linear subspaces in
$\Hom(\Z^n,\R^d)\cong\R^{nd}$, and in particular an arrangement of real (or
complex) affine hyperplanes when $d=1$ ($d=2$, resp.).

When $\GG=S^1$ or $\GG=\C^\times$ and $\Gamma=\Z^n$, we can identify
$\Hom(\Z^n,\GG)\cong\GG^n$ with a real or complex torus, and thus $\A$ describes
a collection of translated subtori.

When $\GG$ is a complex elliptic curve, such arrangements are called elliptic
arrangements.
\end{remark}

\begin{definition}\label{def:layers}
A \textbf{layer} of an arrangement $\A$ is a connected component of an
intersection $\cap_{H\in \B} H$ for some subset
$\B\subseteq\A$. The \textbf{poset of layers} is the set $P(\A)$ whose elements
are the layers of $\A$ partially ordered by reverse inclusion.
\end{definition}

\begin{remark}
By convention, $\Hom(\Gamma,\GG)$ is the unique minimal element of $P(\A)$, and
the atoms of $P(\A)$ may be identified with the elements of $\A$. 
\end{remark}

\begin{definition}
Let $\A$ be an abelian arrangement and $H_0\in\A$. The \textbf{deletion}
of $H_0$ is the abelian arrangement $\A\setminus\{H_0\}$, and
the \textbf{restriction} to $H_0$ is the abelian arrangement $\A^{H_0}$ given by
the nonempty connected components of the intersections 
$H_0\cap H$ where $H\in\A\setminus\{H_0\}$.
\end{definition}

\begin{remark}
It may not be obvious that the restriction is an abelian arrangement by our
\Cref{def:arr}, but it follows from the assumption that $H_0=H_{\alpha,t}$ is
connected. This allows us to fix a direct sum
decomposition $\Gamma\cong(\Z\alpha)\oplus\Gamma_\alpha$ for some
$\Gamma_\alpha\subseteq\Gamma$. 
If $e$ is the identity of $\GG$, then we may identify $H_{\alpha,e}\cong
\Hom(\Gamma_\alpha,\GG)$.
Now, for $(\beta,u)\in\Gamma\times\GG$, we can identify $\beta$ with
$(c\alpha,\beta')\in(\Z\alpha)\oplus\Gamma_\alpha$, and hence the intersection
$H_{\alpha,e}\cap H_{\beta,u}$ may be identified with
$H_{\beta',u}\subseteq\Hom(\Gamma_\alpha,\GG)\cong H_{\alpha,e}$.
Moreover, $H_{\alpha,e}$ is a subgroup of $\Hom(\Gamma,\GG)$ and
$H_{\alpha,t}$ is a translation of this subgroup, explicitly
$H_{\alpha,t}=\phi_{\alpha,t}+H_{\alpha,e}$ where for
$(c\alpha,\beta)\in(\Z\alpha)\oplus\Gamma_\alpha\cong\Gamma$ we have
$\phi_{\alpha,t}(c\alpha,\beta)=ct$.
This translation allows us to view $H_{\alpha,t}$ as a group isomorphic to
$\Hom(\Gamma_\alpha,\GG)$ and intersections $H_{\alpha,t}\cap H_{\beta,u}$ as
translated subgroups of $\Hom(\Gamma_\alpha,\GG)$ as desired.

Note that since each layer $X\in P(\A)$ is connected, we may view it as a
translation of a subgroup of $\Hom(\Gamma,\GG)$ isomorphic to
$\Hom(\Gamma_X,\GG)$ for some 
$\Gamma_X\subseteq\Gamma$. Thus, we can extend the definition of restriction to
any layer, but we omit the details here for simplicity. The operation of
restricting to a layer can also be viewed as an iteration of the operation of
the defined restriction operation.
\end{remark}

\begin{definition}\label{def:localarr}
Let $\A$ be an abelian arrangement (defined by $\X\subseteq\Gamma\times\GG$) and
$X\in P(\A)$ a layer, and let $\X_X=\{\alpha\in\Gamma \st (\alpha,t)\in\X \text{
and } H_{\alpha,t}\supseteq X \text{ for some } t\in\GG\}.$ 
The \textbf{localization} of $\A$ at $X$ is the collection of linear
subspaces $H_{\alpha,0}\subseteq\Hom(\Gamma,\R^{\dim\G})$ with $\alpha\in\X_X$.
\end{definition}

\begin{remark}
Topologically, we can view the localization as an arrangement of linear
subspaces in the tangent space $T_\varphi\Hom(\Gamma,\GG)$ for a generic point
$\varphi\in X$. The linear subspaces are given by $T_\varphi H$ where
$H\in\at(P(\A)_{\leq X})$.
\end{remark}

\begin{example}\label{ex:toric}
Let $\GG=S^1$, $\Gamma=\Z^2$, and $\X=\{((1,1),1),((1,-1),1)\}$. 
Identifying $\Hom(\Z^2,S^1)\cong
S^1\times S^1$ and letting $H_0=H_{(1,-1),1}$, we depict the subtori in the
arrangements $\A$,  $\A^{H_0}$, and $\A\setminus\{H_0\}$, along with the linear
arrangement $\A_{(-1,-1)}$, in \Cref{fig:A,fig:A^H,fig:A-H,fig:A_x},
and the matroid scheme associated to each in
\Cref{fig:MA,fig:MA/H,fig:MA-H,fig:MA_x}.

\begin{figure}[ht]
\phantom{X}
\begin{subfigure}[t]{.2\textwidth}
\centering
\begin{tikzpicture}[scale=1.5]
\draw[-] (0,0)--(1,0)--(1,1)--(0,1)--(0,0);
\draw[ultra thick,-] (0,0)--(1,1);
\draw[ultra thick,-] (0,1)--(1,0);
\solidnodes
\node at (0,0) {};
\node at (.5,.5) {};
\end{tikzpicture}
\caption{$\A$}
\label{fig:A}
\end{subfigure}
\begin{subfigure}[t]{.2\textwidth}
\centering
\begin{tikzpicture}[scale=1.5]
\draw[dashed,-] (0,0)--(1,0)--(1,1)--(0,1)--(0,0);
\draw[ultra thick,-] (0,0)--(1,1);
\draw[dashed,-] (0,1)--(1,0);
\solidnodes
\node at (0,0) {};
\node at (.5,.5) {};
\end{tikzpicture}
\caption{$\A^{H_0}$}
\label{fig:A^H}
\end{subfigure}
\hfill
\begin{subfigure}[t]{.2\textwidth}
\centering
\begin{tikzpicture}[scale=1.5]
\draw[-] (0,0)--(1,0)--(1,1)--(0,1)--(0,0);
\draw[ultra thick,-] (0,1)--(1,0);
\end{tikzpicture}
\caption{$\A\setminus\{H_0\}$}
\label{fig:A-H}
\end{subfigure}
\hfill
\begin{subfigure}[t]{.2\textwidth}
\centering
\begin{tikzpicture}[scale=1.5]
\draw[ultra thick,<->] (0,0)--(1,1);
\draw[ultra thick,<->] (0,1)--(1,0);
\solidnodes
\node at (.5,.5) {};
\end{tikzpicture}
\caption{$\A_{(-1,-1)}$}
\label{fig:A_x}
\end{subfigure}
\phantom{XX}

\bigskip

\begin{subfigure}[t]{.15\textwidth}
\centering
\begin{tikzpicture}[scale=.8]
\node at (-1.3,1.2) {$a$};
\labelnodes
\node (a) at (-1,1.5) {$1$};
\node (b) at (1,1.5) {$1$};
\node (u) at (-1,3) {$2$};
\node (v) at (1,3) {$2$};
\node (0) at (0,0.3) {$0$};
\draw[-] (a)--(0)--(b)--(u)--(a)--(v)--(b);
\end{tikzpicture}
\caption{$\M(\A)$}
\label{fig:MA}
\end{subfigure}
\begin{subfigure}[t]{.25\textwidth}
\centering
\begin{tikzpicture}[scale=.8]
\labelnodes
\node (a) at (-1,1.5) {$0$};
\node (u) at (-1,3) {$1$};
\node (v) at (1,3) {$1$};
\opennodes
\node (b) at (1,1.5) {};
\node (0) at (0,0.3) {};
\draw[dashed,-] (a)--(0)--(b)--(u)--(a)--(v)--(b);
\draw[ultra thick,-] (u)--(a)--(v);
\end{tikzpicture}
\caption{$\M(\A)_{/H_0}\cong\M(\A^{H_0})$}
\label{fig:MA/H}
\end{subfigure}
\begin{subfigure}[t]{.3\textwidth}
\centering
\begin{tikzpicture}[scale=.8]
\labelnodes
\node (b) at (1,1.5) {$1$};
\node (0) at (0,0.3) {$0$};
\opennodes
\node (a) at (-1,1.5) {};
\node (u) at (-1,3) {};
\node (v) at (1,3) {};
\draw[dashed,-] (a)--(0)--(b)--(u)--(a)--(v)--(b);
\draw[ultra thick,-] (0)--(b);
\end{tikzpicture}
\caption{$\M(\A)-H_0\cong\M(\A\setminus\{H_0\})$}
\label{fig:MA-H}
\end{subfigure}
\begin{subfigure}[t]{.2\textwidth}
\centering
\begin{tikzpicture}[scale=.8]
\opennodes
\node (u) at (-1,3) {};
\labelnodes
\node (a) at (-1,1.5) {$1$};
\node (b) at (1,1.5) {$1$};
\node (v) at (1,3) {$2$};
\node (0) at (0,0.3) {$0$};
\draw[-,ultra thick] (b)--(v)--(a)--(0)--(b); 
\draw[-,dashed] (b)--(u)--(a);
\end{tikzpicture}
\caption{$\M(\A)_x\cong\M(\A_x)$}
\label{fig:MA_x}
\end{subfigure}

\caption{A toric arrangement $\A$ along with a restriction, deletion,
and localization, as well as the associated matroid schemes.}
\label{fig:toric}
\end{figure}
\end{example}

\begin{theorem}\label{thm:arr}
To any abelian arrangement $\A$, one may assign a matroid scheme
$\M(\A)$ so that there is an isomorphism of posets \[P(\A)\cong F(\M(\A))\] 
identifying the layers of $\A$ with the flats of $\M(\A)$. Moreover, 
for $H_0\in\A$ and $X\in P(\A)$
one has  the following isomorphisms of matroid schemes
\[
\M(\A\setminus\{H_0\})\cong\M(\A)-H_0; \quad
\M(\A^{H_0})\cong M(\A)_{/H_0}; \quad \text{and} \quad
\M(\A_{X})\cong \M(\A)_{X}.
\]
\end{theorem}
\begin{proof}
As described in \cite[Section 9]{dd},
one can lift an abelian arrangement $\A$ to a(n infinite) periodic
arrangement of affine linear subspaces in the universal cover
$\R^{(\rank\Gamma)(\dim\GG)}$ of $\Hom(\Gamma,\GG)\cong\GG^{\rank\Gamma}$, 
whose intersection data forms a semimatroid $(\cx,\rk)$ with a
cofinite and translative action of a free abelian group $G$. The quotient
$\M/G=(\cx/G,\rk_G)$ from \Cref{thm:Gsemi} is then the desired matroid scheme whose
flats are the layers of $\A$. 

The matroid scheme we associate to an arrangement in this way will always be simple,
hence determined by its poset of flats via \Cref{thm:GP=MS}. 
Then, using the just established correspondence between layers of an arrangement
and flats of its matroid scheme, we have:
\begin{align*}
F(\M(\A\setminus\{H_0\}))&\cong P(\A\setminus\{H_0\}) 
\\ &=  \{w\in P(\A)\st w\in\bigvee A\st
A\subseteq\A\setminus\{H_0\}\}\cong
F(\M(\A)-H_0)
\end{align*}
implying $\M(\A\setminus\{H_0\})\cong\M(\A)-H_0$;
\[ F(\M(\A^{H_0})) \cong P(\A^{H_0}) = P(\A)_{\geq H_0} \cong F(\M(\A)_{/H_0}) \]
implying $\M(\A^{H_0})\cong \M(\A)_{/H_0}$; and
\[ F(\M(\A_X)) \cong P(\A_X) \cong P(\A)_{\leq X} \cong F(\M(\A)_X)\]
implying $\M(\A_X)\cong\M(\A)_X$.
\end{proof}

\begin{corollary}[{\cite[Corollary 4.4.6]{BD}}]
The poset of layers of an abelian arrangement is a geometric poset.
\end{corollary}

\begin{remark}
If the abelian Lie group $\GG$ is noncompact, then the Poincar\'e polynomial for
the complement of an arrangement $\A$ is determined by the characteristic
polynomial of its poset of layers \cite[Theorem 7.8]{LTY}. Via
\Cref{thm:charpoly,thm:arr}, this means the
Poincar\'e polynomial is determined by the Tutte polynomial of the associated
matroid scheme.
\end{remark}

\section{Dowling posets}\label{sec:dowling}

Geometric posets sometimes describe the intersection data of submanifolds in a
manifold more general than the abelian Lie groups that we considered in the
preceding section.
Dowling posets were defined in \cite{BG} to describe the combinatorial structure
associated to an orbit configuration space -- an equivariant analogue of a
configuration space. These posets generalize partition and Dowling lattices,
which are fundamental examples of geometric lattices.

To define these posets, let us fix a positive integer $n$, a finite group $G$,
and a finite $G$-set $T$. Denote $[n]:=\{1,2,\dots,n\}$. Given a subset
$B\subseteq[n]$, a $G$-coloring is a function $b:B\to G$. Define an equivalence
relation on $G$-colorings of $B$ where $(b:B\to G)\sim(b':B\to G)$ whenever
$b'=bg$ for some $g\in G$.
A partial $G$-partition of $[n]$ is a collection
$\beta=\{(B_1,\ol{b_1}),\dots,(B_\ell,\ol{b_\ell})\}$
where $\{B_1,\dots,B_\ell\}$ is a partition of some subset $U\subseteq[n]$ and
each $\ol{b_i}$ is a chosen equivalence class of $G$-colorings on $B_i$.

Let $D_n(G,T)$ be the set of pairs $(\beta,z)$ where $\beta$ is a partial
$G$-partition of $[n]$ and $z:Z_\beta\to T$ is an $T$-coloring of 
$Z_\beta:=[n]\setminus\cup_i B_i$.
This set is partially ordered via the covering relations:
\begin{itemize}
\item $(\beta\cup\{(A,\ol{a}),(B,\ol{b})\},z)\prec (\beta\cup\{(A\cup
B),\ol{a\cup bg})\},z)$ whenever $g\in G$, and 
\item $(\beta\cup\{(B,\ol{b})\},z)\prec(\beta,z')$ whenever $z':B\cup Z_\beta\to
T$ satisfies $z'|_{Z_\beta}=z$ and $z'|_B=f\circ b$ for some $G$-equivariant
function $f:G\to T$.
\end{itemize}
When $|T|=1$, $D_n(G,T)$ is a Dowling lattice.

The Dowling poset $D_n(G,T)$ is bounded below by the trivial partition of $[n]$
into singletons, and it is ranked with $\rk(\beta,z)=n-\ell(\beta)$.
The atoms of $D_n(G,T)$ are of the following two forms:
\begin{itemize}
\item $a_i^t=(\alpha_i^t,z_i^t)$ where $i\in[n]$, $t\in T$, $\alpha_i^t$ is the
trivial partitition of $[n]\setminus\{i\}$, and $z_i^t(i)=t$;
\item $a_{ij}(g) = (\alpha_{ij}(g),\emptyset)$ where $i,j\in[n]$ are distinct,
$g\in G$, and $\alpha_{ij}(g)$ is the partition of $[n]$ whose only nonsingleton
block is $B=\{i,j\}$ with $b(i)=gb(j)$.
\end{itemize}
It was shown in \cite[Theorem A]{BG} that $D_n(G,T)$ is locally geometric -- in
fact every closed interval is isomorphic to a product of partition and Dowling
lattices.

\begin{theorem}\label{thm:dowling}
For any positive integer $n$, finite group $G$, and finite $G$-set $T$, 
the Dowling poset $D_n(G,T)$ is a geometric poset.
\end{theorem}
\begin{proof}
Using \cite[Theorem A]{BG}, we need only check \eqref{g2}.
Let $x\in D_n(G,T)$, $A\subseteq\at(D_n(G,T))$, and
$y\in\bigvee A$, such that $\rk(x)<\rk(y)=|A|$. 
Suppose for contradiction that for every $a\in A$, either $a\leq x$ or $a\vee
x=\emptyset$.
If $a=a_i^t$, then this implies $a_i^u\leq x$ for some $u\in T$. If
$a=a_{ij}(g)$, then this implies either (1) $a_{ij}(h)\leq x$ for some $h\in G$
or (2) $a_i^u\leq x$ for some $u\in T$.
Then we can write 
\[A=\{a_{i_1j_1}(g_1),\dots,a_{i_rj_r}(g_r),a_{i_{r+1}j_{r+1}}(g_{r+1}),\dots,
a_{i_s,j_s}(g_s),a_{i_{s+1}}^{t_{s+1}},\dots,a_{i_m}^{t_m}\}\]
where $m=\rk(y)$, $i_1,\dots,i_m$ are distinct, and there exist
$h_1,\dots,h_r\in G$ and $u_{r+1},\dots,u_m\in T$ with
\[z\in a_{i_1j_1}(h_1)\vee\cdots\vee a_{i_rj_r}(h_r)\vee a_{i_{r+1}}^{u_{r+1}}
\vee\cdots\vee a_{i_m}^{u_m}\]
such that $z\leq x$. 
However, this implies $\rk(x)\geq \rk(z)=m=\rk(y)$, contradicting 
$\rk(x)<\rk(y)$.
\end{proof}

\begin{example}\label{ex:dowling}
There are two actions of $G=\Z_2$ on the set $T=\{\pm1\}$, one trivial and one
nontrivial. When $n=2$, these give rise to the two Dowling posets which are
depicted on the top of \Cref{fig:dowling}, and their corresponding matroid schemes
are depicted on the bottom of the same figure. 

The characteristic polynomial of a Dowling poset does
not depend on the action of $G$ on $T$, in fact it only depends on the integers
$n$, $|G|$, and $|T|$. This means, for instance, that the two Dowling posets in
\Cref{fig:dowling} have the same characteristic polynomial which is
$\chi(t) = (t-2)(t-4)$.
On the other hand, their matroid schemes have different Tutte polynomials, which are
\[T_{\M_1}(\x,\y) = \x^2+4\x+3+4\y+2\y^2 
\hspace{1cm} T_{\M_2}(\x,\y) = \x^2+4\x+1+4\y.\]

\end{example}

\begin{figure}[ht]
\begin{tikzpicture}[scale=.95]
\solidnodes
\node (0) at (0,0.2) {};
\foreach \x in {1,2,3,4,5,6} {
\node (\x) at (\x-3.5,1.5) {};
}
\node (++) at (-2,3) {};
\node (--) at (2,3) {};
\node (+-) at (-.75,3) {};
\node (-+) at (.75,3) {};
\draw[-]
(1)--(++)--(2)--(0)--(3)--(++)--(4)--(0)--(5)--(--)--(6)--(0)--(1)--(+-)--(5);
\draw[-] (2)--(-+)--(6);
\draw[-] (3)--(--)--(4);
\end{tikzpicture}
\hspace{3cm}
\begin{tikzpicture}[scale=.95]
\solidnodes
\node (0) at (0,0.2) {};
\foreach \x in {1,2,3,4,5,6} {
\node (\x) at (\x-3.5,1.5) {};
}
\node (++) at (-2,3) {};
\node (--) at (2,3) {};
\node (+-) at (-.75,3) {};
\node (-+) at (.75,3) {};
\draw[-] (1)--(0)--(6)--(--)--(5)--(0)--(4)--(+-)--(1)--(++)--(2)--(0)--(3)--(++);
\draw[-] (2)--(-+)--(4);
\draw[-] (3)--(--);
\draw[-] (6)--(-+);
\draw[-] (+-)--(5);
\end{tikzpicture}

\vspace{2mm}

\begin{tikzpicture}[scale=.9]
\tikzstyle{every node}=[draw,circle,inner sep=1pt,scale=.8]
\node (0) at (0,0.2) {\footnotesize 0};
\foreach \x in {1,2,3,4,5,6} {
\node (\x) at (\x-3.5,1.5) {\footnotesize 1};
\draw[-] (0)--(\x);
\node (a\x) at (0.5*\x-4,3) {\footnotesize 2};
\node (b\x) at (0.5*\x+0.5,3) {\footnotesize 2};
}
\node (u) at (-2.25,5.8) {\footnotesize 2};
\node (v) at (2.25,5.8) {\footnotesize 2};
\foreach \x in {1,2,3,4} {
\node (c\x) at (0.5*\x-3.5,4.5) {\footnotesize 2};
\node (d\x) at (0.5*\x+1,4.5) {\footnotesize 2};
\draw[-] (c\x)--(u);
\draw[-] (d\x)--(v);
}
\foreach \x in {1,2} {
\draw[-] (c\x)--(a1)--(\x);
\draw[-] (d\x)--(b1);
}
\foreach \x in {1,3} {
\draw[-] (c\x)--(a2)--(\x);
\draw[-] (d\x)--(b2);
}
\foreach \x in {2,3} {
\draw[-] (c\x)--(a3)--(\x);
\draw[-] (d\x)--(b3);
}
\foreach \x in {1,4} {
\draw[-] (c\x)--(a4)--(\x);
\draw[-] (d\x)--(b4);
}
\foreach \x in {2,4} {
\draw[-] (c\x)--(a5)--(\x);
\draw[-] (d\x)--(b5);
}
\foreach \x in {3,4} {
\draw[-] (c\x)--(a6)--(\x);
\draw[-] (d\x)--(b6);
}
\draw[-] (3)--(b1)--(4)--(b4)--(5)--(b2)--(3)--(b3)--(6)--(b6)--(5);
\draw[-] (6)--(b5)--(4);
\node (+-) at (-.25,3) {\footnotesize 2};
\node (-+) at (.25,3) {\footnotesize 2};
\draw[-] (1)--(+-)--(5);
\draw[-] (2)--(-+)--(6);
\end{tikzpicture}
\hspace{1cm} 
\begin{tikzpicture}[scale=.9]
\tikzstyle{every node}=[draw,circle,inner sep=1pt,scale=.8]
\node (0) at (0,0.2) {\footnotesize 0};
\foreach \x in {1,2,3,4,5,6} {
\node (\x) at (\x-3.5,1.5) {\footnotesize 1};
\draw[-] (0)--(\x);
}
\node (a) at (-3,4.5) {\footnotesize 2};
\node (b) at (-1,4.5) {\footnotesize 2};
\node (c) at (1,4.5) {\footnotesize 2};
\node (d) at (3,4.5) {\footnotesize 2};
\foreach \x in {1,2,3} {
\node (a\x) at (0.5*\x-4,3) {\footnotesize 2};
\draw[-] (a\x)--(a);
}
\draw[-] (a1)--(1)--(a2)--(3)--(a3)--(2)--(a1);
\foreach \x in {1,2,3} {
\node (b\x) at (0.5*\x-2,3) {\footnotesize 2};
\draw[-] (b\x)--(b);
}
\draw[-] (b1)--(1)--(b2)--(6)--(b3)--(5)--(b1);
\foreach \x in {1,2,3} {
\node (c\x) at (0.5*\x,3) {\footnotesize 2};
\draw[-] (c\x)--(c);
}
\draw[-] (c1)--(2)--(c2)--(6)--(c3)--(4)--(c1);
\foreach \x in {1,2,3} {
\node (d\x) at (0.5*\x+2,3) {\footnotesize 2};
\draw[-] (d\x)--(d);
}
\draw[-] (d1)--(3)--(d2)--(6)--(d3)--(5)--(d1);
\end{tikzpicture}
\caption{The two Dowling posets (top) of \Cref{ex:dowling} and their corresponding matroid schemes (bottom).}
\label{fig:dowling}
\end{figure}

\end{document}